\documentclass[12pt]{amsart}
\usepackage{amscd}
\usepackage{amssymb,amsmath, amsthm}
\usepackage{thmtools}
\usepackage{hyperref}
\usepackage{mathrsfs}
\usepackage{graphicx}
\usepackage{enumitem}
\usepackage{romannum}
\usepackage{mathtools}
\usepackage[utf8]{inputenc}
\usepackage{listings}
\usepackage[capitalize]{cleveref}

\usepackage{lipsum,eso-pic,xcolor}
\usepackage{lineno}

\usepackage{tikz}
\usetikzlibrary{calc, arrows.meta, decorations.markings}
\usetikzlibrary{arrows,decorations.pathmorphing,backgrounds,positioning,fit}
\usetikzlibrary{positioning}
\hypersetup{
    colorlinks=true,
    linkcolor=blue,
    hypertexnames=false,
}

\usetikzlibrary{trees}
\usetikzlibrary{arrows}

\def\supp{\operatorname{supp}}

\def\deg{\operatorname{deg}}

\newcommand{\NN}{\mathbb{N}}

\newcommand{\G}{\mathcal{G}}

\newtheorem{lemma}{Lemma}[section]
\newtheorem{corollary}[lemma]{Corollary}
\newtheorem{theorem}[lemma]{Theorem}
\newtheorem{proposition}[lemma]{Proposition}
\newtheorem{definition}[lemma]{Definition}
\newtheorem{remark}[lemma]{Remark}
\newtheorem{notation}[lemma]{Notation}
\newtheorem{example}[lemma]{Example}

\newtheorem{setting}[lemma]{Setting}

\advance\headheight1.15pt

\newtheorem{innercustomthm}{Theorem}
\newenvironment{customthm}[1]
  {\renewcommand\theinnercustomthm{#1}\innercustomthm\itshape}
  {\endinnercustomthm}

\newtheorem{innercustomcor}{Corollary}
\newenvironment{customcor}[1]
{\renewcommand\theinnercustomcor{#1}\innercustomcor\itshape}
{\endinnercustomcor}

\begin{document}
	
	\pagenumbering{arabic}
	
	\title[Homological shift ideals]{Homological shift ideals of weighted oriented graphs} 
	
	\author[M. Kumar]{Manohar Kumar}
	\address{Department of Mathematics, Indian Institute of Technology
		Madras, Chennai, INDIA - 600036.}
	\email{manhar349@gmail.com}

    \author[J. Mondal]{Joydip Mondal$^a$}
	\address{Department of Mathematics, Indian Institute of Technology
		Kharagpur, West Bengal, INDIA - 721302.}
	\email{joydipmondal1999i@gmail.com}
	
	\author[R. Nanduri]{Ramakrishna Nanduri$^b$}
	\address{Department of Mathematics, Indian Institute of Technology
		Kharagpur, West Bengal, INDIA - 721302.}
  \email{nanduri@maths.iitkgp.ac.in}

        \thanks{$^a$ Supported CSIR-UGC PhD Fellowship, India}
        \thanks{$^b$ corresponding author}

	\thanks{AMS Classification 2020: 13D02, 16E05, 05E40}

	\maketitle
	
\begin{abstract}
    In this paper, we study the homological shift ideals of edge ideals associated with weighted oriented graphs. For a weighted oriented graph $D$, let $HS_k(I(D))$ denote the $k^{th}$ homological shift ideal of its edge ideal $I(D)$. If $D$ is vertex-splittable, then we characterize that $HS_1(I(D))$ has linear quotients if and only if $D_6$, $D_7$, and $D_8$ are not induced subgraphs of $D$. Furthermore, we show that if $I(D)$ has linear quotients, then $\sqrt{HS_k(I(D))} = HS_k(I(G))$, for all $k\geq 1$, where $G$ is the underlying simple graph of $D$. We show that if $I(D)$ has homological linear quotients, then $I(G)$ also has homological linear quotients. If $D$ is a tree, then we establish the following characterization:     
    \begin{align*}
    HS_k(I(D)) \text{ has linear quotients for all } k\geq 0 \iff ~ &G~ \text{is}  \text{ a star graph or a broom graph} \\ &\text{ and }~ D \text{ is $D_i$-free, for } i=1,2,5,6,8.
   \end{align*} 
	\end{abstract}
    
\section{Introduction}
The $\mathbb{N}^n$-graded minimal free resolution of a monomial ideal is a fundamental object in combinatorial commutative algebra, providing deep insight into the algebraic and homological properties of the ideal. Its multigraded structure reflects the underlying combinatorial features of the generators, while the syzygy modules encode the relations among them at successive homological levels. Understanding the structure of these syzygies and the organization of the associated graded minimal free resolution has therefore been a central theme of research, leading to significant interactions between commutative algebra, combinatorics, and topology. 

Let $\mathbb{K}$ be a field and $I$ be a monomial ideal in the polynomial ring 
$R = \mathbb{K}[x_1, \ldots, x_n]$. The $\mathbb{N}^n$-graded minimal free resolution of 
$I$ takes the form of an exact sequence
\[
\mathcal{F}_{\bullet} : 0 \to F_r \xrightarrow{\partial_r} \cdots 
\xrightarrow{\partial_2} F_1 \xrightarrow{\partial_1} F_0 
\xrightarrow{\partial_0} I \to 0,
\]
where $F_k = \displaystyle \bigoplus_j R^{\beta_{k,\mathbf{a}_{kj}}(I)}(-\mathbf{a}_{kj})$ for each $k \geq 0$, and $R(-\mathbf{a}_{kj})$ denotes the polynomial ring $R$ equipped with a grading shift by the integer vector $\mathbf{a}_{kj}$. The vectors $\mathbf{a}_{kj}$ 
are referred to as the $k$th \textit{multigraded shifts} of $I$, while the 
integer $r$ is known as the \textit{projective dimension} of $I$, written 
$\mathrm{pd}(I)$. The positive integers $\beta_{k,\mathbf{a}_{kj}}(I)$ are called the  $k$th-{\it multigraded Betti numbers} of $I$. 

Motivated by the goal of systematically understand multigraded shifts $\mathbf{a}_{kj}$ in 
minimal free resolution $\mathcal{F}_{\bullet}$, the notion of \emph{homological shift ideals} has emerged as a fruitful direction of research. Although the underlying ideas appeared implicitly in the work of Miller and Sturmfels \cite{MS05} in 2005, the explicit construction of ideals generated by multigraded shifts was introduced later by Bayati \emph{et al.} \cite{b18, bjt19}. The subject gained significant momentum following the influential work of Herzog \emph{et al.}\ \cite{hmmz21}, who formalized the terminology homological shift ideal and established the concept as a distinct object of study. Since then, homological shift ideals have attracted considerable attention due to their close connections with the homological and combinatorial aspects of monomial ideals.

For a non-negative integer $k$, the $k$th \textit{homological shift ideal} 
of $I$, denoted $\mathrm{HS}_k(I)$, is defined as the monomial ideal generated by all monomials of the form ${\bf x}^{\mathbf{a}_{kj}}$, where each $\mathbf{a}_{kj}$ ranges over the $k$th multigraded shifts of $I$. That is, for each $k\geq 0$, 
$$\mathrm{HS}_k(I) := ({\bf x}^{\mathbf{a}_{kj}} ~:~ \beta_{k,\mathbf{a}_{kj}}(I)\neq 0).$$
In particular, one has 
$\mathrm{HS}_0(I) = I$, and $\mathrm{HS}_k(I) = (0)$ for all $k \geq \mathrm{pd}(I)+1$. It is natural to investigate which combinatorial and homological properties are shared by the homological shift ideals $HS_k(I)$, for $k=0,\ldots,\text{pd}(I)$. Any property satisfied by every $\mathrm{HS}_k(I)$ is referred to as a \emph{homological shift property} of $I$. In particular, if each $HS_k(I)$ has linear quotients, then $I$ is said to have \emph{homological linear quotients}. 

Homological shift ideals have been studied for a broad spectrum of monomial ideals, including polymatroidal ideals \cite{s24, b18, s23, f25}, ideals of Borel type \cite{bjt19, f25}, edge ideals of graphs, bounded principal Borel ideal, and some families of monomial ideals with linear quotients \cite{cdm26, fh23, fa25, fq25, hmmz21}, and vertex cover ideals \cite{cf23, cf24}, among others. A subtle but fundamental feature of the theory is that the minimal generating set of  $\mathrm{HS}_k(I)$ does not, in general, capture all multigraded shifts occurring in the $k$th free module of the minimal free resolution of $I$. Consequently, monomial ideals with linear quotients constitute a particularly important class in the study of homological shift ideals. 

In this paper, our goal is to study the homological shift ideals of edge ideals $I(D)$ of (vertex) weighted oriented graphs $D$, with a particular focus on determining exactly when $HS_k(I(D))$ has linear quotients. We identify the exact combinatorial obstructions, expressed as forbidden induced subgraphs of $D$, that characterize the homological linear quotients of $I(D)$. We first investigate the linear quotient property of $HS_1(I(D))$ of \emph{vertex-splittable} weighted oriented graphs $D$, a class defined inductively through vertex splittings and designed to capture edge ideals with good homological properties. In this setting, we establish the following result: 

\begin{customthm}{\ref{vertex splittable}}
Let $D$ be a vertex-splittable weighted oriented graph as in \Cref{s1s2s3s4}. Then $HS_1(I(D))$ has linear quotient property if and only if $D_6, D_7, D_8$ are not induced subgraphs of $D$ as in \Cref{fig7}.
\end{customthm}

This theorem shows that, within the vertex-splittable class, the linear quotient property of $HS_1(I(D))$ is governed entirely by the avoidance of three small weighted oriented graphs $D_6, D_7, D_8$ as induced weighted oriented subgraphs a phenomenon analogous to classical forbidden-subgraph characterizations for edge ideals of simple graphs.

We then turn to the relationship between the homological shift ideals of $I(D)$ and those of the edge ideal $I(G)$ of the underlying unweighted, undirected graph $G$. A priori, orienting the edges of $G$ and assigning weights can substantially change the algebraic invariants of the resulting ideal; nevertheless, we show that when $I(D)$ has the linear quotient property, its homological shift ideals are radically determined by those of $I(G)$:

\begin{customthm}{\ref{T-radicalhomological}}
Let $D$ be a weighted oriented graph with the underlying simple graph $G$ such that $I(D)$ has linear quotient property. Then
$$\sqrt{HS_k(I(D))} = HS_k(I(G)) \;\; \text{for all } k \geq 1.$$
\end{customthm}

As an immediate consequence, the linear-quotient behavior of the homological shift ideals descends from $D$ to its underlying graph $G$:

\begin{customcor}{\ref{HS_K(G)}}
Let $D$ be a weighted oriented graph such that $I(D)$ has homological linear quotients. Then $I(G)$ has homological linear quotients.
\end{customcor}

This corollary reduces certain questions about weighted oriented graphs to the (often better understood) unweighted, undirected case, and provides a useful necessary condition: if the underlying graph $G$ fails to have homological linear quotients, then no orientation or weight of the vertex of $G$ can produce a weighted oriented graph $D$ with this property. As a consequence, we give a necessary condition for the edge ideal of a weighted oriented graph to have homological linear quotients (\Cref{complement}). Furthermore, \Cref{T-radicalhomological} yields that several important algebraic properties of $HS_k(I(D))$ are inherited by those of the underlying simple graph $G$ of $D$: if $HS_k(I(D))$ is Cohen-Macaulay or Gorenstein or sequentially Cohen–Macaulay \emph{etc.}, then $HS_k(I(G))$ is so (\Cref{sec4cor1}). 

Finally, we specialize to the case where the underlying graph $G$ is a tree, obtaining a complete characterization. Combining the forbidden-subgraph obstructions from our first result with a structural classification of the underlying simple graph, we prove:

\begin{customthm}{\ref{thm:hmologicalshift:tree}}
  Let $D$ be a weighted oriented tree with its underlying simple graph $G$. Then 
  $HS_k(I(D))$ has linear quotients for all $k\geq 0$ \;$\iff$  \;$G$ is a star graph or a broom graph as in \Cref{fig9} and $D_i$ are not induced subgraphs of $D$ for all $i\in \{1,2,5,6,8\}$.
\end{customthm}

This theorem gives a complete, purely combinatorial description of when \emph{all} homological shift ideals of $I(D)$ simultaneously have linear quotients in the tree case: the underlying tree must be one of exactly two shapes a star or a broom and the orientation and weights must avoid a specific finite list of forbidden induced weighted oriented subgraphs $D_1, D_2, D_5, D_6, D_8$.

The paper is organized as follows. In \Cref{sec2}, we review the definitions and preliminary results needed to establish our main theorems. In \Cref{sec3}, we characterize the linear quotient property of $HS_1(I(D))$. Finally, in \Cref{sec4}, we investigate the linear quotient property of the higher homological shift ideals  $HS_k(I(D))$ and provide a characterization for weighted oriented trees.  

\section{Preliminaries} \label{sec2}
In this section, we recall some necessary prerequisites, which are used to describe our work and establish our results.\par 
\noindent
Let $D=(V(D), E(D), w)$ be a (vertex) weighted oriented graph with $V(D)=\{x_1,\ldots,x_n\}$ and underlying simple graph $G$. An edge $e\in E(D)$ is an ordered pair $e=(x_i,x_j)$ with $x_i,x_j\in V(D)$ and the orientation of $e$ is from the vertex $x_i$ to the vertex $x_j$. Also, $w:V(D)\rightarrow \NN$ is the weight function which assigns a weight for each vertex of $D$. Throughout the paper, the weight of the vertex $x_i$ is denoted by $w(x_i)$, and we abbreviate $w(x_i)$ by $w_i$. We write $V^+(D):= \{x\in V(D): w(x)>1 ~~\text{and}~~x ~\text{is not a  source vertex of $D$}\}$, in short $V^{+}(D)$ is denoted by $V^+$. For a vertex $x\in V(D)$, its {\it outer neighbourhood} is defined as 
  $N_{D}^{+}(x):= \{y \in V(D) | (x,y)\in E(D)\}$, its {\it inner neighbourhood} is defined as $N_{D}^{-}(x):= \{z \in V(D) | (z,x)\in E(D)\}$, and $N_D(x):=N_{D}^{+}(x) \cup N_{D}^{-}(x)=N_{G}(x)$. A vertex $x\in V(D)$ is called a {\it source} if $N_{D}^{-}(x)=\emptyset$ and $x$ is called a {\it sink} if $ N_{D}^{+}(x) = \emptyset $. If $x\in V(D)$ is a source, then we may assume $w(x)=1$ as it would not affect the ideal $I(D)$. For a subset $A\subseteq V(D)$, we denote the induced subgraph of $D$ on $A$ by $D[A]$. If $x\in V(D)$ is a vertex, then we write $D\setminus x$ to denote the graph $D[V(D)\setminus \{x\}]$.
  
 Now, let us discuss some notations and definitions regarding simple graphs. A graph $G$ is said to be a {\it complete graph} if there is an edge between each pair of vertices of $G$ and a complete graph on $n$ vertices is denoted by $K_{n}$. A {\it clique} of a graph $G$ is a set $A$ of vertices of $G$ such that $G[A]$ is a complete graph. A set of vertices $C$ of $G$ is called a {\it vertex cover} of $G$ if $C\cap e\neq \emptyset$ for all $e\in E(G)$. A {\it minimal vertex cover} is a vertex cover that is minimal with respect to inclusion. The \textit{complement} of a simple graph $G$, denoted by $G^c$, is the simple graph such that $V(G^c)=V(G)$ and $E(G^c)=\{\{u,v\}\mid \{u,v\}\not\in E(G)\}$. A cycle of length $n$, denoted by $C_n$, is a connected graph such that every vertex of $C_n$ has degree two. A simple graph $G$ is called {\it chordal} if $G$ has no induced cycle of length greater than three. A graph $G$ is called {\it co-chordal} if $G^c$ is chordal. A weighted oriented graph $D$ is said to be a cycle or complete or chordal or co-chordal if its underlying simple graph is so. 
 
Throughout this paper, let $R=\mathbb{K}[x_1,\ldots,x_n]$ be the polynomial ring in $n$ variables over a field $\mathbb{K}$ and $I$ be a monomial ideal of $R$. We denote by $\mathcal{G}(I)$ the unique minimal set of monomial generators of $I$ and by $\sqrt{I}$ the radical of $I$. For a monomial $u=x_1^{a_1}\cdots x_n^{a_n}$, the \emph{support} of $u$ is defined by $\supp(u)=\{\,x_i\mid a_i>0\,\}$.

\begin{definition}{\rm
    A monomial ideal $I \subseteq R$ has linear quotient property if there exists an order $ u_1 < \cdots < u_m$ on the minimal monomial generating set $\mathcal{G}(I)=\{u_1,\ldots,u_m\}$ of $I$ such that the colon ideal $((u_1, \ldots, u_{i-1}) : u_i)$ is generated by a subset of the variables, for $i = 2, \ldots, m$.
    }
\end{definition}

\begin{lemma}\cite[Lemma 2.1]{jz10}\label{L-increasing-admissible}
Let $I \subseteq R$ be a monomial ideal with linear quotients. Then there is a degree increasing admissible order of $\mathcal{G}(I)$ such that $I$ has linear quotients with respect to this degree increasing order.
\end{lemma}

\begin{lemma}\cite[Lemma 3.2]{fmr25}\label{sum}
  Let $x\in X$ be a variable, and let $I_1\subseteq R=\mathbb{K}[X]$ and
$I_2\subseteq \mathbb{K}[X\setminus\{x\}]$ be monomial ideals with linear quotients such that
$I_2\subseteq I_1$. Suppose that $ \mathcal{G}(xI_1)\subseteq \mathcal{G}(I) $.
Then $ I = xI_1 + I_2$ has linear quotients.
\end{lemma}

\begin{definition}\label{spl1}{\rm 
A monomial ideal $I\subseteq R=\mathbb{K}[X]$ is called {\it vertex splittable} if it can be obtained by the following recursive procedure.
\begin{enumerate}
    \item If $v$ is a monomial and $I=(v)$, $I=(0)$ or $I=R$, then $I$ is vertex splittable.
    \item If there is a variable $x$ in $R$ and vertex splittable ideals $I_1$ and $I_2$ in $\mathbb{K}[X\setminus \{x\}]$ such that $I=xI_1+I_2$, $I_2 \subseteq I_1 $ and $\mathcal{G}(I)= \mathcal{G}(xI_1)\sqcup \mathcal{G}(I_2)$, then $I$ is a vertex splittable. For $I=xI_1+I_2$, the variable $x$ is said to be a {\it splitting variable} for $I$.
\end{enumerate}
}
\end{definition}

\begin{definition}
A monomial ideal $I$ of $R$ is called \emph{weakly polymatroidal} if for every two monomials $v=x_1^{b_1}\cdots x_n^{b_n}<_{\mathrm{lex}}u=x_1^{a_1}\cdots x_n^{a_n}$
belonging to $\mathcal{G}(I)$ such that $a_1=b_1,\;\ldots,\;a_{t-1}=b_{t-1}\quad\text{and}\quad a_t>b_t$ for some $t$, there exists $j>t$ such that $x_t\left(\frac{v}{x_j}\right)\in I$.
\end{definition}
Note that, in the above definition, the lexicographic monomial order
$<_{\mathrm{lex}}$ on $R$ is induced by the ordering $x_1>x_2>\cdots>x_n$
of the variables.

The following proposition follows from \cite[Proposition 1.7]{cf23} and \cite[Proposition 2.7]{hmmz21}.

\begin{proposition}\label{betti split}
 Let $I=xI_1+I_2$ be a vertex splittable ideal. Then
$$HS_k(I)=xHS_k(I_1)+xHS_{k-1}(I_2)+HS_k(I_2), ~\text{for all} ~k\geq 0.$$   
\end{proposition}

Let $I$ be a monomial ideal with $\mathcal{G}(I)=(u_1,\ldots,u_m)$. Suppose that $I$ has quotients with respect to the ordering $u_1<\cdots<u_m$.
Now, we denote
\[
\operatorname{set}_I(u_j) := \left\{\, k \in [n] \;\middle|\; x_k \in (u_1, \ldots, u_{j-1}) : u_j \right\}, \quad j = 1, \ldots, m.
\]

\begin{lemma}\cite[Lemma 1.5]{ht02} \label{L-homologicalshift}
Let $I$ be a monomial ideal in $R$ having linear quotients. Then    \[
\mathrm{HS}_k(I)
=
\left( x_F\,u \;\middle|\; u \in \G(I),\; |F| = k,\; F \subseteq \operatorname{set}_I(u) \right).
\]
\end{lemma}
Note that the support of all minimal generators of $HS_k(I)$ has the same cardinality.

\begin{lemma}\cite[Lemma 3.1]{kmn25}\label{V+}
     Let $D$ be a weighted oriented graph. If $D_i$ are not induced subgraphs of $D$ for $i\in \{1,\ldots,4\}$,  then  $|N_D^+(x)\cap V^+|\leq 1$, for all $x\in V(D)$.
\end{lemma}

For each $n\geq 6$, let $H_n$ be the graph with vertex set $V(H_n)=\{x_1,\ldots,x_n\}$
and edge set
\begin{align*}
E(H_n)= {}&
\{x_ix_{i+1}: i\in[n-5]\} 
{}\cup \{x_{n-4}x_{n-2},\,x_{n-2}x_{n-3}\} \\
&{}\cup \{x_{n-1}x_j: j\in[n]\setminus\{1,n-1\}\} 
{}\cup \{x_nx_j: j\in[n]\setminus\{n-3,n\}\}.
\end{align*}

\begin{theorem}\cite[Theorem 5.1]{cdm26}\label{Hs_k(I(G))}
If the edge ideal $I(G)$ has homological linear quotients, then $G$ is co-chordal and $H_n^c$-free for every $n\geq 6$.
\end{theorem}

\section{First Homological shift of edge ideals of weighted oriented graphs} \label{sec3}
In this section, we study the first homological shift ideal $HS_1(I(D))$. We first show that, if $HS_1(I(D))$ has linear quotients, then certain weighted oriented graphs cannot be induced subgraphs of $D$. We then prove that, for a vertex-splittable weighted oriented graph $D$, $HS_1(I(D))$ has linear quotients if and only if $D$ is $\mathcal{F}$-free, where $\mathcal{F}$ is the family of weighted oriented graphs as in \Cref{fig7}.


\begin{lemma}\label{linear quotients}
  Let $I \subseteq R$ be a monomial ideal. For a variable $x \in R$, consider the ideal $I'$ such that $\mathcal{G}(I')=\mathcal{G}(I)\setminus \{\,u\in \mathcal{G}(I): x\mid u\,\}$.
Then $I'$ has linear quotients if $I$ has linear quotients.
\end{lemma}
\begin{proof}
Let $I=(u_1,\ldots,u_m)$ has linear quotients with respect to the order $u_1<\cdots <u_m$. Assume that $I'=(u_{i_1},\ldots,u_{i_k})$ with $u_{i_1}<\cdots <u_{i_k}$.

\noindent
\underline{Claim:} $u_{i_1}<\cdots <u_{i_k}$ be a linear quotients order of $I'$.

Let $u_{i_m},u_{i_r}\in \mathcal{G}(I')$ with $u_{i_m}<u_{i_r}$. If $u_{i_m}:u_{i_r}$ is a variable, then nothing to show. Assume $u_{i_m}:u_{i_r}$ is not a variable. Then we will show that there exists $i_{l} \geq i_1$ such that $i_{l}<i_r$ and $u_{i_l}:u_{i_r}$ is a variable which divides $u_{i_m}:u_{i_r}$. Since $I$ has a linear quotients then there exist $j<i_r$ such that $u_{j}:u_{i_r}$ is a variable, which divides $u_{i_m}:u_{i_r}$. If $u_j\in \mathcal{G}(I')$, then we are done. Assume $u_j\notin \mathcal{G}(I')$ this means that $x\mid u_j$. Since $u_j:u_{i_r}$ is a variable and
$x\in \operatorname{Supp}(u_j)\setminus \operatorname{Supp}(u_{i_r})$, we have
$u_j:u_{i_r}=x$. Consequently, $x$ divides $u_{i_m}:u_{i_r}$, contradicting the fact that
$u_{i_m}\in \mathcal{G}(I')$. Therefore, $u_j\in \mathcal{G}(I')$, which completes the proof. 
\end{proof}

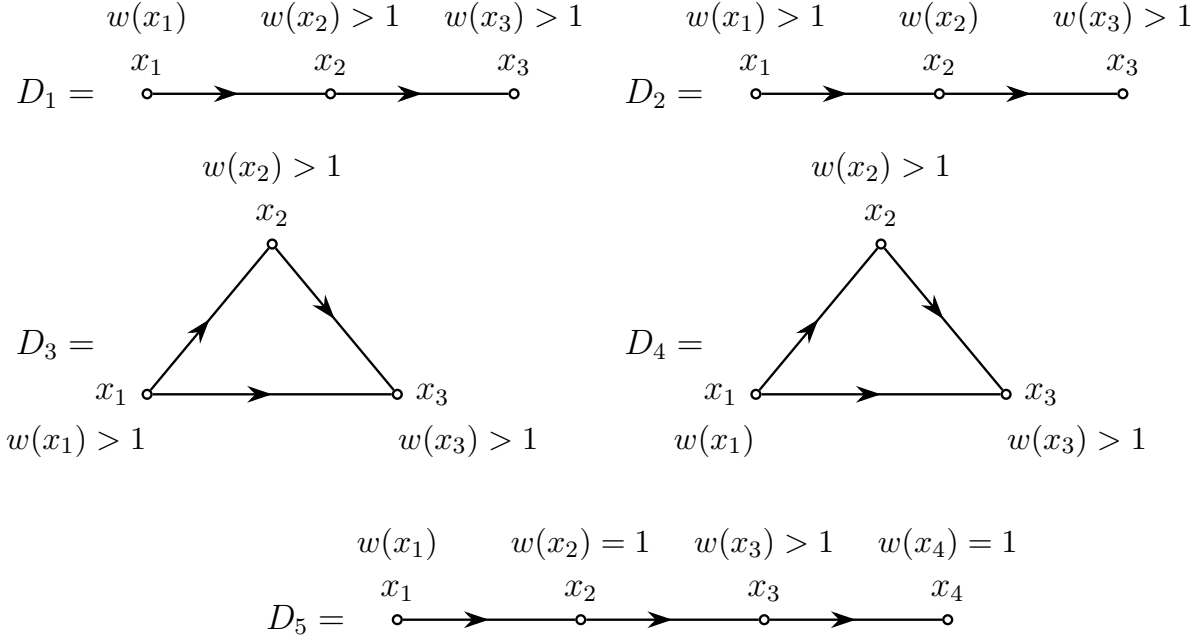
\begin{figure}[htbp]
\centering
\resizebox{\textwidth}{!}{%
\begin{tikzpicture}[
    vertex/.style={circle,draw,fill=white,minimum size=3pt,inner sep=0pt},
    every node/.style={font=\normalsize},
    wlabel/.style={font=\small},
    midarrow/.style={
        postaction={decorate},
        decoration={markings, mark=at position 0.5 with {\arrow{Stealth[length=3mm,width=2mm]}}}
    },
    midarrowrev/.style={
        postaction={decorate},
        decoration={markings, mark=at position 0.5 with {\arrow{Stealth[length=3mm,width=2mm]}}}
    },
    thick
]
\node at (-1.1,0) {$D_1=$};
\node[vertex] (d1-1) at (0,0) {};
\node[vertex] (d1-2) at (2.2,0) {};
\node[vertex] (d1-3) at (4.4,0) {};
\draw[midarrow] (d1-1) -- (d1-2);
\draw[midarrow] (d1-2) -- (d1-3);
\node[above=2pt] at (d1-1) {$x_1$};
\node[above=2pt] at (d1-2) {$x_2$};
\node[above=2pt] at (d1-3) {$x_3$};
\node[wlabel,above=16pt] at (d1-1) {$w(x_1)$};
\node[wlabel,above=16pt] at (d1-2) {$w(x_2)>1$};
\node[wlabel,above=16pt] at (d1-3) {$w(x_3)>1$};

\begin{scope}[xshift=7.3cm]
\node at (-1.1,0) {$D_2=$};
\node[vertex] (d2-1) at (0,0) {};
\node[vertex] (d2-2) at (2.2,0) {};
\node[vertex] (d2-3) at (4.4,0) {};
\draw[midarrowrev] (d2-1) -- (d2-2);
\draw[midarrow] (d2-2) -- (d2-3);
\node[above=2pt] at (d2-1) {$x_1$};
\node[above=2pt] at (d2-2) {$x_2$};
\node[above=2pt] at (d2-3) {$x_3$};
\node[wlabel,above=16pt] at (d2-1) {$w(x_1)>1$};
\node[wlabel,above=16pt] at (d2-2) {$w(x_2)$};
\node[wlabel,above=16pt] at (d2-3) {$w(x_3)>1$};
\end{scope}

\begin{scope}[yshift=-3.6cm]
\node at (-1.1,0.6) {$D_3=$};
\node[vertex] (d3-1) at (0,0) {};
\node[vertex] (d3-2) at (1.5,1.8) {};
\node[vertex] (d3-3) at (3,0) {};
\draw[midarrow] (d3-1) -- (d3-2);
\draw[midarrow] (d3-2) -- (d3-3);
\draw[midarrowrev] (d3-1) -- (d3-3);
\node[left=2pt] at (d3-1) {$x_1$};
\node[above=2pt] at (d3-2) {$x_2$};
\node[right=2pt] at (d3-3) {$x_3$};
\node[wlabel,above=16pt] at (d3-2) {$w(x_2)>1$};
\node[wlabel,below left=6pt and -4pt] at (d3-1) {$w(x_1)>1$};
\node[wlabel,below right=6pt and -4pt] at (d3-3) {$w(x_3)>1$};
\end{scope}

\begin{scope}[xshift=7.3cm,yshift=-3.6cm]
\node at (-1.1,0.6) {$D_4=$};
\node[vertex] (d4-1) at (0,0) {};
\node[vertex] (d4-2) at (1.5,1.8) {};
\node[vertex] (d4-3) at (3,0) {};
\draw[midarrow] (d4-1) -- (d4-2);
\draw[midarrowrev] (d4-2) -- (d4-3);
\draw[midarrow] (d4-1) -- (d4-3);
\node[left=2pt] at (d4-1) {$x_1$};
\node[above=2pt] at (d4-2) {$x_2$};
\node[right=2pt] at (d4-3) {$x_3$};
\node[wlabel,above=16pt] at (d4-2) {$w(x_2)>1$};
\node[wlabel,below left=6pt and -4pt] at (d4-1) {$w(x_1)$};
\node[wlabel,below right=6pt and -4pt] at (d4-3) {$w(x_3)>1$};
\end{scope}

\begin{scope}[xshift=3cm,yshift=-6.3cm]
\node at (-1.1,0) {$D_5=$};
\node[vertex] (d5-1) at (0,0) {};
\node[vertex] (d5-2) at (2.2,0) {};
\node[vertex] (d5-3) at (4.4,0) {};
\node[vertex] (d5-4) at (6.6,0) {};
\draw[midarrow] (d5-1) -- (d5-2);
\draw[midarrow] (d5-2) -- (d5-3);
\draw[midarrow] (d5-3) -- (d5-4);
\node[above=2pt] at (d5-1) {$x_1$};
\node[above=2pt] at (d5-2) {$x_2$};
\node[above=2pt] at (d5-3) {$x_3$};
\node[above=2pt] at (d5-4) {$x_4$};
\node[wlabel,above=16pt] at (d5-1) {$w(x_1)$};
\node[wlabel,above=16pt] at (d5-2) {$w(x_2)=1$};
\node[wlabel,above=16pt] at (d5-3) {$w(x_3)>1$};
\node[wlabel,above=16pt] at (d5-4) {$w(x_4)=1$};
\end{scope}
\end{tikzpicture}%
}
  \caption{$D_1,D_2,D_3,D_4,D_5$ can not be induced subgraph of $D$ whose edge ideal has linear quotients.}
    \label{fig1}
\end{figure}

\begin{corollary}\label{corollary-2..5}
    Let $D$ be a weighted oriented graph. Suppose $I(D)$ has linear quotients. Then $D_1, D_2, D_3,D_4$ and $D_5$ as in Figure \ref{fig1} cannot be induced subgraphs of $D$.
\end{corollary}
\begin{proof} By \cite[Corollary 3.3]{kns25}, we have $D_1, D_2, D_3,D_4$ cannot be induced subgraphs of $D$.
  Let $D_5$ be a weighted oriented graph as in Figure \ref{fig1}. Then $I(D_5)=(x_1x_2,x_2x_3^{w_3},x_3x_4)$. It is easy to see that $I(D_5)$ does not have linear quotients for any order of its generators. By \Cref{linear quotients}, we have $D_5$ cannot be an induced subgraph of $D$. 
\end{proof}

For a monomial ideal $I$ and a monomial $m$, we denote by $I^{\le m}$ the ideal generated by all minimal monomial generators of $I$ that divide $m$. In particular, $I^{\le m}$ is a monomial subideal of $I$. 
\begin{proposition}\label{I(D')}
    Let $D$ be a weighted oriented graph such that $I(D)$ has linear quotients. Let $D'$ be an induced weighted oriented subgraph of $D$. Let $k\geq 1$. If $HS_k(I(D))$ has linear quotients, then $HS_k(I(D'))$ has linear quotients.
\end{proposition}
\begin{proof}
    By \cite[Corollary 2.10]{hmmz21}, we have $HS_k(I(D'))=HS_k(I(D)^{\leq \prod_{x_i\in V(D')}x_i^{w(x_i)}}$. Let $f\in \mathcal{G}(HS_k(I(D))$. Then by \Cref{L-homologicalshift}, we have $f=x_Fu$ for some $F \subseteq \text{set}_{I(D)}(u)$ and $u\in \mathcal{G}(I(D))$. Then $$f\in HS_k(I(D')) \iff x_Fu \; | \;   \prod_{x_i\in V(D')}x_i^{w(x_i)} \iff \supp(f) \subseteq V(D').$$  
This implies that 
\begin{eqnarray*}
\mathcal{G}(HS_k(I(D'))) &=& \mathcal{G}(HS_k(I(D))) \setminus \{ g \in \mathcal{G}(HS_k(I(D))) : \supp(g) \nsubseteq V(D') \} \\ 
&=& \mathcal{G}(HS_k(I(D))) \setminus \{ g \in \mathcal{G}(HS_k(I(D))) : x_i | g \text{ for some }  x_i \in V(D) \setminus V(D') \} \\
 &=& \mathcal{G}(HS_k(I(D))) \setminus \cup_{x_i \in V(D) \setminus V(D')} \{ g \in \mathcal{G}(HS_k(I(D))) : x_i | g  \}. 
\end{eqnarray*}
Then by \Cref{linear quotients}, we have $HS_k(I(D'))$ has linear quotients.    
\end{proof}

\begin{proposition}\label{not induced}
Let $ D_6, D_7,D_8$ be weighted oriented graphs as in \Cref{fig7}. Then $HS_1(I(D_i))$ does not have linear quotients for $i\in\{6,7,8\}$. 
  \end{proposition}
\begin{proof}
Let $I(D_6)=(x_2x_3,x_1x_2^{w(x_2)},x_3x_4^{w(x_4)})$, $I(D_7)=(x_3x_1,x_1x_2^{w(x_2)},x_2x_3^{w(x_3)})$ and $I(D_8)=(x_2x_3,x_2x_1^{w(x_1)},x_3x_4^{w(x_4)})$. By \Cref{L-homologicalshift}, we have 
\begin{align*}
  HS_1(I(D_6))=(x_3x_1x_2^{w(x_2)},x_2x_3x_4^{w(x_4)})\\
  HS_1(I(D_7))=(x_3x_1x_2^{w(x_2)},x_1x_2x_3^{w(x_3)})\\
  HS_1(I(D_8))=(x_3x_2x_1^{w(x_1)},x_2x_3x_4^{w(x_4)})
\end{align*}
 It is clear that none of these ideals have linear quotients. 
\end{proof}


\begin{figure}[htbp]
\centering
\resizebox{\textwidth}{!}{%
\begin{tikzpicture}[
    vertex/.style={circle,draw,fill=white,minimum size=3pt,inner sep=0pt},
    every node/.style={font=\normalsize},
    wlabel/.style={font=\small},
    midarrow/.style={
        postaction={decorate},
        decoration={markings, mark=at position 0.5 with {\arrow{Stealth[length=4mm,width=3.5mm]}}}
    },
    midarrowrev/.style={
        postaction={decorate},
        decoration={markings, mark=at position 0.5 with {\arrow{Stealth[length=3mm,width=2mm] reversed}}}
    },
    thick
]
\node (D6label) at (-1.5,0) {$D_6=$};
\node[vertex] (d6-1) at (0,0) {};
\node[vertex] (d6-2) at (3,0) {};
\node[vertex] (d6-3) at (6,0) {};
\node[vertex] (d6-4) at (9,0) {};

\draw[midarrow] (d6-1) -- (d6-2);
\draw[midarrow] (d6-2) -- (d6-3);
\draw[midarrow] (d6-3) -- (d6-4);

\node[above=3pt] at (d6-1) {$x_1$};
\node[above=3pt] at (d6-2) {$x_2$};
\node[above=3pt] at (d6-3) {$x_3$};
\node[above=3pt] at (d6-4) {$x_4$};

\node[wlabel,above=22pt] at (d6-2) {$w(x_2)>1$};
\node[wlabel,above=22pt] at (d6-3) {$w(x_3)=1$};
\node[wlabel,above=22pt] at (d6-4) {$w(x_4)>1$};

\begin{scope}[yshift=-3.5cm]
\node at (-1.5,0) {$D_8=$};
\node[vertex] (d8-1) at (0,0) {};
\node[vertex] (d8-2) at (3,0) {};
\node[vertex] (d8-3) at (6,0) {};
\node[vertex] (d8-4) at (9,0) {};

\draw[midarrow] (d8-2) -- (d8-1);
\draw (d8-2) -- (d8-3);
\draw[midarrow] (d8-3) -- (d8-4);

\node[above=3pt] at (d8-1) {$x_1$};
\node[above=3pt] at (d8-2) {$x_2$};
\node[above=3pt] at (d8-3) {$x_3$};
\node[above=3pt] at (d8-4) {$x_4$};

\node[wlabel,above=22pt] at (d8-1) {$w(x_1)>1$};
\node[wlabel,above=22pt] at (d8-2) {$w(x_2)=1$};
\node[wlabel,above=22pt] at (d8-3) {$w(x_3)=1$};
\node[wlabel,above=22pt] at (d8-4) {$w(x_4)>1$};
\end{scope}

\begin{scope}[xshift=13cm,yshift=-3 cm]
\node at (-0.8,1.7) {$D_7=$};
\node[vertex] (d7-1) at (0,0.2) {};
\node[vertex] (d7-2) at (1.8,3.0) {};
\node[vertex] (d7-3) at (3.8,0.2) {};

\draw[midarrow] (d7-1) -- (d7-2);
\draw[midarrow] (d7-2) -- (d7-3);
\draw[midarrow] (d7-3) -- (d7-1);

\node[left=3pt] at (d7-1) {$x_1$};
\node[above=3pt] at (d7-2) {$x_2$};
\node[right=3pt] at (d7-3) {$x_3$};

\node[wlabel,above=22pt] at (d7-2) {$w(x_2)>2$};
\node[wlabel,below left=10pt and -6pt] at (d7-1) {$w(x_1)=1$};
\node[wlabel,below right=10pt and -6pt] at (d7-3) {$w(x_3)>2$};
\end{scope}
\end{tikzpicture}%
}
 \caption{$D_6,D_7,D_8$ can not be  induced subgraphs of $D$ whose homological shift ideals $HS_k(I(D))$ have linear quotients for all $k\geq 0$.}
    \label{fig7}
\end{figure}
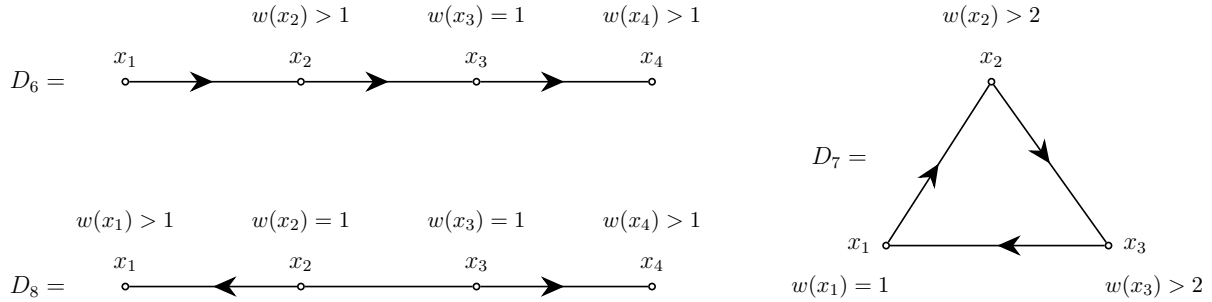

\begin{corollary}\label{D_i not induced}
      Let $D$ be a weighted oriented graph such that the edge ideal $I(D)$ has linear quotients. Suppose $HS_1(I(D))$ has linear quotients. Then $D_i$ (as in \Cref{fig7}) are not induced subgraphs of $D$, for $i=6,7,8$. 
\end{corollary}
\begin{proof} It follows from \Cref{I(D')} and \Cref{not induced}.
\end{proof}

\noindent
To characterize when the first homological shift ideal of a vertex-splittable weighted oriented graph has linear quotients, we first establish the following setting and notation, which will be used throughout the subsequent lemmas and the proof of the main theorem of this section.
\begin{setting}\label{s1s2s3s4}
Let $D$ be a weighted oriented graph with the vertex sets $V(D)=\{x_1,\ldots,x_n\}$ such that $I=I(D)$ is vertex-splittable with the splitting  $I=x_1I_1+I_2$, where $x_1$ is a variable and $I_2\subseteq I_1$. Note that $I_2=I(D\setminus x_1) \subseteq I_1$ gives that $\sqrt{I_1}=(N_D(x_1))$ and $N_D(x_1)$ is a minimal vertex cover of $D$. Since vertex splittable ideals have linear quotients, it follows from \Cref{corollary-2..5}, we have that $D_i$ are not induced subgraphs for $i\in\{1,\ldots,5\}$. Therefore by \Cref{V+}, $|N_D^+(x_1) \cap V^+|\leq 1$, then we can write $I_1=(x_{m+1},\ldots,x_\ell^{w(x_\ell)})$, for some $m$, where $w(x_\ell)\geq 1$ and $N_D(x_1)=\{x_{m+1},\ldots,x_\ell\}$. By \Cref{spl1}, $I_2$ is vertex splittable. Therefore, $I_2$ has linear quotients. Then by \Cref{L-increasing-admissible}, let $I_2=(g_1,\ldots,g_s)$ with $g_1,\ldots,g_s$ be a degree increasing linear quotient order of $I_2$. By \Cref{betti split}, we have 
\begin{equation}\label{eq-1}
  HS_1(I)=x_1(HS_1(I_1)+I_2)+HS_1(I_2)  
\end{equation}
Let $X=\{g_i \in \mathcal{G}(I_2) \mid g_i \not \in HS_1(I_1)\}= \{g_{i_1},\ldots,g_{i_p}\} \subseteq \{g_1,\ldots,g_s\}$ with $i_1<\dots<i_p$.  Note that  
    $\text{set}_{I_1}(x_i)=\{m+1,\ldots,i-1\}$ for all $i\in\{m+2,\ldots,\ell-1\}$ and $\text{set}_{I_1}(x_\ell^{w_\ell})=\{m+1,\ldots,{\ell-1}\}$. Then using \Cref{L-homologicalshift}, we get 
    $$HS_1(I_1)=(x_ix_j\mid x_i, x_j \in N_D(x_1)\setminus \{x_l\})+(x_ix_\ell^{w_\ell}\mid x_i \in N_D(x_1)\setminus \{x_l\}).$$    
Let
\begin{align*}
S' :=\,& \{g_{i_j} \in X \mid g_{i_j}=x_rx_{\ell}^{w_{\ell}} \text{ for some } x_r \in V(D)\setminus N_D(x_1)\},\\[2mm]
S'' :=\,& \{x_rx_\ell^{w_\ell}\in\mathcal G(HS_1(I_1))
          \mid x_r\in N_D(x_1)\setminus\{x_\ell\}\}, \\[2mm]
S_1 :=\,& \{x_ix_j\mid x_i, x_j \in N_D(x_1)\setminus \{x_{\ell}\}\}
       =\{u_1,\ldots,u_\alpha\}, (\mbox{say})\\[2mm]
S_2 :=\,& X \setminus S' = \{g_{i_j}\in X \mid 
           g_{i_j}=x_tx_l^{w_l} \text{ for some } x_t \in N_D(x_1) \text{ or } l\neq \ell\} \\
       =\,& \{g_{j_1},\ldots,g_{j_q}\},
       \quad \mbox{with } j_1<\cdots<j_q, (\mbox{say})\\[2mm]
       S_3 :=\,& \{x_rx_\ell^{w_\ell}\in S''
          \mid
          x_rx_\ell^{w_\ell}\notin (S_2)\}\\[2mm]
S_4 :=\,& S'
       =\{x_rx_\ell^{w_\ell}\in X
         \mid x_r\in V(D)\setminus N_D(x_1)\}.
\end{align*}
Then $HS_1(I_1)+I_2= (S_1\cup S_2\cup S_3\cup S_4)$. Moreover, $S_1\cup S_2\cup S_3\cup S_4$ is the minimal generating set of $HS_1(I_1)+I_2$ because no monomial in this union divides another. Furthermore, $I_2 \subseteq HS_1(I_1)+(S_2\cup S_4)$. 
\end{setting}

\begin{remark}\label{subset}
Suppose $x_ax_b^{w_b}\in I_2$. Let $a=\ell$ and $w_\ell\geq 2$. Then $x_\ell x_b^{w_b}\in I_2\subseteq I_1=(x_{m+1},\ldots,x_\ell ^{w_\ell})$. Since $w_\ell \geq 2$, this gives that $x_c \mid  x_\ell x_b^{w_b}$ for some $c \in \{m+1,\ldots,\ell-1\}$. This implies that $c=b$ which gives $x_b\in N_D(x_1)\setminus \{x_\ell\}$. 
\end{remark}
\noindent
We first consider the case $w_\ell=1$ and show that $HS_1(I_1)+I_2$ has linear quotients.
\begin{lemma}\label{wl-one} 
Let $D$ be a vertex splittable weighted oriented graph as in Setting~\ref{s1s2s3s4}. If $w_\ell=1$, then $HS_1(I_1)+I_2$ has linear quotients. 
\end{lemma}

\begin{proof}
Assume $w_\ell =1$. Note that $N_D(x_1)=\{x_{m+1},\dots,x_\ell\}$. Then 
\begin{align*}
 HS_1(I_1)=&(x_ix_j\mid x_i, x_j \in N_D(x_1)\setminus \{x_l\})+(x_ix_\ell \mid x_i \in N_D(x_1)\setminus \{x_l\}) \\
 =&(x_ix_j\mid x_i, x_j \in N_D(x_1))=(u_1\ldots,u_{\alpha})+(u_{\alpha+1}, \ldots, u_{\beta}) 
\end{align*} 
which is square-free Veronese in the variables $x_{m+1},\dots,x_\ell$, where $\{u_1\ldots,u_{\alpha}\}=S_1$. By \cite[Corollary 4.2]{hmmz21}, square-free Veronese ideals has linear quotients. Let $HS_1(I_1)$ has linear quotients with respect to the order $u_1,\ldots,u_{\beta}$. Also, we have 
\begin{align*}
HS_1(I_1)+I_2 = &(u_1\ldots,u_{\alpha})+(u_{\alpha+1}, \ldots, u_{\beta})+ (X) \\
=& (u_1\ldots,u_{\alpha})+(u_{\alpha+1}, \ldots, u_{\beta})+(g_{i_1},\ldots,g_{i_p}),
\end{align*}
which is a minimal generating set of $HS_1(I_1)+I_2$ because of degree reason, where $X$ is as in \Cref{s1s2s3s4}.  

Now, we show that this order is a linear quotient order for $HS_1(I_1)+I_2$. Note that $u_1,\dots,u_{\beta}$ is already in linear quotient order. Now, we show that $(u_1,\dots,u_{\beta}): g_{i_l}$ is generated by some variables. If $u_r:g_{i_{l}}$ is a variable for $u_r\in \mathcal{G}(HS_1(I))$, then there is nothing to show. Suppose $u_r:g_{i_{l}}$ is not a variable. Now, we will show that there exist $u_k\in \mathcal{G}(HS_1(I_1))$ such that $u_k:g_{i_l}$ is a variable which divides $u_r:g_{i_l}$. Let $u_r=x_ix_j$, for some $i,j\in \{m+1,\ldots,\ell\}$ and $g_{i_l}=x_ax_b^{w_b}$ for $1 \leq l \leq p$. Since, $g_{i_l} \in I_2\subseteq I_1=(N_D(x_1))$ then one of $x_a $ or $x_b$ is in $I_1$ because $N_D(x_1)$ is a minimal vertex cover of $D$. Without loss of generality, assume $x_a\in I_1$, that is, $a\in \{m+1,\dots,\ell\}$. Then we get $u_k := x_ax_i\in\mathcal{G}(HS_1(I_1))$ because $HS_1(I_1)$ is square-free Veronese ideals. Then $u_k:g_{i_l}=x_i$ which divides $u_r:g_{i_l}$, as required. It remains to show that $(u_1,\ldots,u_{\beta}, g_{i_1},\dots, g_{i_{j-1}}) : g_{i_{j}}$ is generated by some variables, for any $j$.  If $g_{i_{l}}:g_{i_{j}}$ is a variable for $g_{i_{l}},g_{i_{j}}\in X$, then there is nothing to show. Suppose $g_{i_{l}}:g_{i_{j}}$ is not a variable. Since $g_1,\ldots,g_s$ is a linear quotient order of $I_2$ then there exist a $t<i_j$ such that $g_t:g_{i_l}= z$ (say) is a variable, which divides $g_{i_{l}}:g_{i_{j}}$. If $t\in \{i_1,\ldots,i_{j-1}\}$, then we are done. If $t\notin \{i_1,\ldots,i_{j-1}\}$, then $g_t\in HS_1(I_1)$. This implies that $u_k\mid g_t$ for some $u_k\in \mathcal{G}(HS_1(I_1))$. Then we get $u_k:g_{i_j}$ divides  $g_t:g_{i_l}=z$ and $z$ divides $g_{i_{l}}:g_{i_{j}}$ which implies that $u_k:g_{i_j}=z$ is a variable, as required. Thus, $u_1,\ldots,u_{\beta},g_{i_1},\ldots,g_{i_p}$ is a linear quotient order for $HS_1(I_1)+I_2$. 
\end{proof}

We now consider the case $w_\ell\ge 2$. We first establish the following lemmas to prove $HS_1(I_1)+I_2$ has linear quotients.
\begin{lemma}\label{wl-two} 
Let $D$ be a vertex-splittable weighted oriented graph as in  Setting~\ref{s1s2s3s4}. Assume $w_{\ell} \geq 2$. Then 
\begin{enumerate}
\item[(1)] $((S_1):g_{j_i})$ is generated by variables for any $j_i$. 
\item[(2)] For any $g_{j_l}, g_{j_i} \in S_2$, either $g_{j_l}:g_{j_i}$ is a variable, or there exists $u_k \in S_1$ such that $u_k: g_{j_i}$ is a variable which divides $g_{j_l}:g_{j_i}$.  
\item[(3)] $(S_1 \cup S_2)$  has a linear quotient order.
\end{enumerate}
\end{lemma}
\begin{proof}
 Since $S_1$ is square-free Veronese in the variables. By \cite[Corollary 4.2]{hmmz21}, square-free Veronese ideals has linear quotients. Let $(S_1)=(u_1\ldots,u_{\alpha})$ has a linear quotient order, and $S_2=\{g_{j_1},\ldots,g_{j_q} \}$. 
 \vskip 0.2cm
\noindent
  (1) Consider $u_l:g_{j_i}$, for any $1\leq l\leq \alpha$.  If $u_l:g_{j_i}$ is variable, then there is nothing to prove. Assume  $u_l:g_{j_i}$ is not variable. Set $u_l=x_ix_j\in S_1$ and $g_{j_i}=x_ax_b^{w_b}\in S_2$ such that $u_l:g_{j_i}$ is not a variable. Now, we will show that there exists $u_k\in S_1$  such that $u_k:g_{j_i}$ is a variable which divides $u_l:g_{i_j}$. Since $g_{j_i}=x_ax_b^{w_b}$ is an edge in $D$ and $N_D(x_1)$ is a minimal vertex cover of $D$, then one of $x_a $ or $x_b$ is in $N_D(x_1)=\{x_{{m+1}}, \dots,x_{\ell}\}$. Now, we see in three cases: 

 Case (i): If $x_a\in N_D(x_1)\setminus \{x_\ell\}$ (respectively, $x_b\in N_D(x_1)\setminus \{x_\ell\}$), then taking $u_k=x_ix_a$ which is in $S_1$ (respectively, $u_k=x_ix_b$ which is in $S_1$) gives that $u_k:g_{j_i}=x_i$, which divides $u_l:g_{j_i}$, as required. 
 
 Case (ii): Assume $a=\ell$, then $g_{j_i}=x_\ell x_b^{w_b}\in S_2\subset I_2$. By \Cref{subset}, we get $x_b\in N_D(x_1)\setminus \{x_\ell\}$. Therefore, we get $x_ix_b\in S_1$. Take  $u_k=x_ix_b\in S_1$. Then $u_k:g_{j_i}=x_i$ which divides $u_l:g_{j_i}$, as required. 
 
 Case (iii): Assume $b=\ell$, then we have $g_{j_i}=x_ax_b^{w_b}=x_ax_\ell^{w_\ell}\in S_2$. This implies that $x_a\in N_D(x_1)$. Therefore, we get $u_k=x_ax_i\in S_1$. This implies that $u_k:g_{j_i}=x_i$ which divides $u_l:g_{j_i}$, as required.

\noindent
 (2) Let $g_{j_l}, g_{j_i}\in S_2$. If $g_{j_l}:g_{j_i}$ is a variable, then there is nothing to show. Suppose $g_{j_l}:g_{j_i}$ is not a variable. Since $g_1,\ldots,g_s$ is a linear quotient order, then there exists a $1\leq t<j_i$ such that $g_t:g_{j_i}$ is a variable, which divides $g_{j_{l}}:g_{j_{i}}$. If $t\in \{j_1,\ldots,j_{i-1}\}$, then we are done. If $g_t\in (S_1)$, then $u_k\mid g_t$, for some $u_k\in S_1$. So $u_k:g_{j_i}$ is a variable because $g_t:g_{j_i}$ is a variable and $u_k:g_{j_i}$ divides $g_t:g_{j_i}$. Since $g_t:g_{j_i}$ divides $g_{j_{l}}:g_{j_{i}}$, then we get $u_k:g_{j_i}$ is a variable, which divides $g_{j_{l}}:g_{j_{i}}$, for some $u_k\in S_1$, as required. Assume $t\notin \{j_1,\ldots,j_{i-1}\}$ and $g_t\notin (S_1)$. That is, $g_t\notin S_1 \cup S_2$. Note that, every generator of $I_2$ is either divisible by an element of $\mathcal{G}(HS_1(I_1))$ or belongs to $S_2\cup S_4$. If $g_t$ divisible by an element of $\mathcal{G}(HS_1(I_1))$. Then either $g_t\in (S_1)$ or $g_t\in S''= \{x_ix_\ell^{w_\ell}\mid x_i \in N_D(x_1)\setminus \{x_l\}\}$. If $g_t\in (S_1)$, we have already shown above. If $g_t\in S''$, then $g_t=x_rx_\ell^{w_\ell}$, for $x_r\in N_D(x_1)\setminus \{x_l\}$. We have $g_t:g_{j_i}$ is a variable and $g_t=x_rx_{\ell}^{w_\ell}$. This implies that $g_t:g_{j_i}=x_r$ and $g_{j_i}=x_ax_\ell^{w_\ell}$, for $x_a\in N_D(x_1)\setminus \{x_l\}$. Therefore, we get $x_ax_r\in S_1$. Take $u_k=x_ax_r$. Hence $u_k:g_{j_i}=x_r$, which divides $g_{j_i}:g_{j_i}$, as required. If $g_t\in (S_2\cup S_4)$, then $g_t\in S_4$ because $g_t\notin(S_1\cup S_2)$. Hence $g_t=x_rx_{\ell}^{w_\ell}$, for some $x_r\in V(D)\setminus N_D(x_1)$. Since $g_t:g_{j_i}$ is a variable and $w_\ell\geq 2$, we have $g_t:g_{j_i}=x_r$. This implies that $g_{j_i}=x_ax_\ell ^{w_\ell}$, for some $x_a \neq x_r \in V(D)$. As $g_{j_i}=x_ax_\ell ^{w_\ell}\in S_2$, it follows that $x_a\in N_D(x_1)\setminus \{x_\ell\}$. Since $x_r$ divides $g_{j_{l}}:g_{j_i}$, there are three possibilities:
\begin{equation*}
   g_{j_l}:g_{j_i}\in \{x_r^{w_r},~x_jx_r^{w_r},~x_rx_j^{w_j}\},~ \text{for some}~~ j. 
\end{equation*}

\item (i)
Assume $g_{j_{l}}:g_{j_i}=x_r^{w_r}$. Since $g_{j_l}:g_{j_i}$ is not a variable, then $w _r \geq 2$. As $g_{j_i}=x_ax_\ell^{w_\ell}$ and $g_{j_{l}}:g_{j_i}=x_r^{w_r}$, it follows that $g_{j_l}=x_ax_r^{w_{r}}$ or $x_\ell x_r^{w_{r}}$. If $g_{j_\ell}=x_\ell x_r^{w_r}$, then together with $g_t=x_r x_\ell^{w_\ell}$, it yields two edges joining the same pair of vertices, contradicting the fact that $G$ is a simple graph. Therefore, $g_{j_l}=x_ax_r^{w_{r}}$. Consequently, $x_{\ell},x_r \in N_D^+(x_a)\cap V^+$,  contradicting \cref{V+}. 
\item(ii)
Assume $g_{j_{l}}:g_{j_i}=x_jx_r^{w_r}$. Then necessarily $g_{j_{l}}=x_jx_r^{w_r}$. Since $x_r\in V(D)\setminus N_D(x_1)$ and $N_D(x_1)$ is a minimal vertex cover of $D$, then $x_j\in N_D(x_1)$. Also, $x_j\neq x_\ell$, otherwise $g_{j_{l}}:g_{j_i}=x_r^{w_r}$, which is a contradiction. Then $x_j\in N_D(x_1)\setminus \{x_\ell\}$. Now we show that $x_a\neq x_j$. Suppose $x_a=x_j$. If $w_r=1$, then we get $g_{j_{l}}:g_{j_i}=x_r$. If $w_r\geq2$, then  $x_{\ell},x_r \in N_D^+(x_a)\cap V^+$,  contradicting \cref{V+}. Therefore we get $x_a\neq x_j$. Thus $x_ax_j\in S_1$, say $u_k=x_ax_j$. This implies that $u_k:g_{j_i}=x_j$, which divides $g_{j_{l}}:g_{j_i}$, as required.
\item(iii)
Assume $g_{j_{l}}:g_{j_i}=x_rx_j^{w_j}$. Then necessarily $g_{j_l}=x_rx_j^{w_j}$. From $g_t=x_rx_\ell
^{w_\ell}$, we have $x_\ell\in N_D^+(x_r)$ and by \Cref{V+}, $|N_D^+(x_r)\cap V^+|\leq 1$ implies $w_j=1$. Then  $g_{j_l}=x_rx_j$. Since $g_{j_l}:g_{j_i}$ is not a variable, then $x_a\neq x_j$. Since $g_{j_l}$ is an edge of $D$ and $x_r\in V(D)\setminus N_D(x_1)$, then $x_j\in N_D(x_1)$. Also, $x_j\neq x_\ell$, otherwise $g_{j_l}=g_t\in S_4$, which contradicting that $g_{j_l}\in S_2$.  Hence, $x_j\in N_D(x_1)\setminus \{x_\ell\}$ and $x_ax_j\in S_1$, denote $u_k=x_ax_j$. This implies that $u_k:g_{j_i}=x_j$, which divides $g_{j_{l}}:g_{j_i}$, as required.
\vskip 0.2cm 
\noindent
(3)  follows from (1) and (2). 
\end{proof}

\begin{lemma}\label{s1-S3-S4-colon}
Let $D$ be a vertex-splittable weighted oriented graph as in \Cref{s1s2s3s4} and assume that
$w_\ell\ge2$. Let $u\in S_1$ and $v\in S_3\sqcup S_4$. Then either $u:v$ is a variable, or there exists
\[
z\in
\begin{cases}
S_1,& \text{if } v\in S_3,\\
S_2\sqcup S_3,& \text{if } v\in S_4,
\end{cases}
\]
such that $z:v$ is a variable which divides $z:v\mid u:v$.
\end{lemma}

\begin{proof} Let $u \in S_1$, say $u=x_ix_j$, for some $i,j\in \{m+1,\ldots,\ell-1\}$. If $u: v$ is a variable, there is nothing to show. Assume $u: v$ is not a variable. We will prove in the following two cases:\\
{ \bf Case 1:} Suppose $v\in S_3$, say $v=x_rx_\ell^{w_\ell}\in S_3$, where $r\in \{m+1,\ldots,\ell -1\}$. This implies that $x_ix_r\in S_1$, denoted by $u'=u_k=x_ix_r$. This implies that $u':v=x_i$, which divides $u:v$.

\noindent 
{\bf Case 2:} Suppose $v\in S_4$, say $v=x_rx_\ell^{w_\ell}\in S_4$ such that $u:v$ is not a variable. From the definition of $S''$, we have $u'=x_ix_\ell^{w_\ell}\in S''$ for some $i$. This implies that,  $u'\in S_3$ or $u'\in (S_2)$. If $u'\in S_3$, then $u'=x_ix_\ell^{w_\ell}$. Therefore we get $u':v=x_i$, which divides $u:v$. If $u'\in (S_2)$, then $u''\mid u'$, for some $u''\in S_2$. This implies that $u'':v\mid u':v=x_i$. Then we get $u'':v=x_i$, which divides $u:v$, as required. 
\end{proof}

\begin{lemma}\label{lem-S_2-S_4}
  Let $D$ be a vertex-splittable weighted oriented graph as in \Cref{s1s2s3s4} and assume that
$w_\ell\ge2$. Let $u\in S_2$ and $v\in S_4$. Then either $u:v$ is a variable, or there exists $z\in S_2\sqcup S_3$
such that $z:v$ is a variable which divides $u:v$.
\end{lemma}

\begin{proof} Let $u=x_ax_b^{w_b}\in S_2$ and $v=x_rx_\ell^{w_\ell}\in S_4$, where $x_r\in V(D)\setminus N_D(x_1)$. If $u:v$ is variable then nothing to show. Assume $u:v$ is not a variable. If $b = \ell$, then it is easy to see that $u:v=x_a$ which is a variable, as required. Assume $b \neq \ell$. Since $u=x_ax_b^{w_b}\in S_2\subset I_2$ which is an edge in $D$ and $N_D(x_1)$ is minimal vertex cover. Then one of $x_a$ or $x_b$ in $N_D(x_1)$.

\item Case(i) Suppose $x_a\in N_D(x_1)\setminus\{x_\ell\}$ (respectively,  $x_b\in N_D(x_1)\setminus\{x_\ell\}$), then $x_a\neq x_r$, because $x_r\in V(D)\setminus N_D(x_1)$. From the definition of $S''$ we get $u'=x_ax_\ell^{w_\ell}\in S''$ (respectively, $u''=x_bx_\ell^{w_\ell}$). This implies that $u'\in S_3$ or $u'\in (S_2)$. If $u'\in S_3$, then $u':v=x_a$, which divides $u:v$, as required. If $u'\in (S_2)$, then $u''\mid u$, for some $u''\in S_2$. This implies that $u'':v\mid u':v=x_a$. Then we get $u'':v=x_a$, which divides $u:v$, as required. 
\item Case(ii)
Suppose $x_\ell=x_a$. Then $x_a\in N_D(x_1)$. Then by \Cref{subset}, we have $x_b\in N_D(x_1)\setminus\{x_\ell\}$. This falls under Case 1 and we are done in this case. 
\end{proof}

\begin{lemma} \label{colon1}
  Let $D$ be a vertex-splittable weighted oriented graph as in \Cref{s1s2s3s4}. Assume that $D_6,D_7,D_8$ are not induced subgraphs of $D$ and $w_\ell\geq2$. Suppose that 
  \begin{enumerate}
      \item $N_D^+(x_\ell)\cap V^+=\{x_t\}\nsubseteq N_D(x_1)$  \; or 
      \item $N_D^+(x_\ell)\cap V^+=\emptyset$ \; or 
      \item $N_D^+(x_\ell)\cap V^+=\{x_t\}\subseteq N_D(x_1)$ and $w_{\ell} > 2$. 
  \end{enumerate}
 For any $u\in S_2$ and $v\in S_3$, either $u:v$ is a variable or there exists $u'\in S_1$ such that $u':v$ is a variable which divides $u:v$.   
\end{lemma}
\begin{proof}
Let $u=x_ax_b^{w_b}\in S_2$ and $v=x_{r}x_\ell^{w_\ell}\in S_3$, where $r\in \{m+1,\ldots,\ell -1\}$. Assume $u:v$ is not a variable. Note that one of $x_a$ or $x_b$ is in $N_D(x_1)=\{x_{m+1},\ldots ,x_\ell\}$ because $N_D(x_1)$ is a vertex cover of $D$. Now we consider the following two cases. 
 
\noindent
\textbf{Case 1:}
Assume $x_b\in N_D(x_1)$. If $b = \ell$, then $u:v=x_a$ which is a variable, as required. Now onwards we assume that $b\neq \ell$. If $b\neq r$, then $x_bx_{r}\in S_1$. Let $u'=x_bx_{r}$, then $u':v=x_b,$ which divides $u:v$, as required. Therefore, it remains to consider the case
$b=r$. Thus $u=x_ax_b^{w_b}$ and $v=x_bx_\ell^{w_\ell}$. Then $w_b>1$, otherwise $u:v=x_a$.
If $a\neq \ell$, then the induced subgraph $D[\{x_1,x_a,x_b,x_\ell\}]$ is same as $D_6$, which is a contradiction. Therefore $a=\ell$. 

(1) Assume $N_D^+(x_\ell)\cap V^+=\{x_t\}\nsubseteq N_D(x_1)$. For $a=\ell$, we get $u=x_\ell x_b^{w_b}\in S_2$ and $v=x_bx_\ell^{w_\ell}\in S_3$. As $u:v$ is not a variable, then $w_b>2$. This gives that $x_b\in N^+_D(x_\ell)\cap V^+$. Therefore, we have $x_b \in N_D^+(x_\ell)\cap V^+=\{x_t\}$. This gives that $b=t$ and hence  $u=x_\ell x_t^{w_t}$ and $v=x_tx_\ell^{w_\ell}$. Thus $v=x_tx_\ell^{w_\ell}\in S_3$. This implies that $x_t \in N_D(x_1)$, which contradicts that $\{x_t\}\nsubseteq N_D(x_1)$. Thus, in either case $a \neq \ell$ or $a=\ell$ we get a contradiction. This implies that $b\neq r$.

(2) Assume $N_D^+(x_\ell)\cap V^+=\emptyset$. For $a=\ell$, we get $u=x_\ell x_b^{w_b}\in S_2$ and $v=x_bx_\ell^{w_\ell}\in S_3$. As $u:v$ is not a variable, then $w_b>2$. This implies that $x_b\in N_D^+(x_\ell)\cap V^+$, which is a contradiction. Thus, in either case $a \neq \ell$ or $a=\ell$ we get a contradiction. This implies that $b\neq r$.

(3) Assume $N_D^+(x_\ell)\cap V^+=\{x_t\}\subseteq N_D(x_1)$ and $w_{\ell} > 2$. For $a=\ell$, we get $u=x_\ell x_b^{w_b}\in S_2$ and $v=x_bx_\ell^{w_\ell}\in S_3$. As $u:v$ is not a variable, then $w_b>2$.  Since $w_\ell>2$, then we get the induced subgraph $D[\{x_1,x_\ell,x_b\}]$ is same as $D_7$, which is a contradiction. Thus, in either case $a \neq \ell$ or $a=\ell$ we get a contradiction. This implies that $b\neq r$.

At the beginning of this case, we already shown that the conclusion of the lemma is true if $b\neq r$.    
 
 \noindent
 \textbf{Case 2:} Assume $x_a\in N_D(x_1)$. If $a=\ell$, then by \Cref{subset}, we get $x_b\in N_D(x_1)\setminus\{x_\ell\}$, which falls under Case 1 and we are done in this case. 
Assume $a\neq \ell$. We first show that $a\neq r$. If $a=r$, then $u=x_ax_b^{w_b}~ \text{and}~ v=x_ax_\ell^{w_\ell}$. Thus $w_b>1$, otherwise $u:v=x_a$. This implies that the induced subgraph $D[\{x_1,x_a,x_b,x_\ell\}]$ is same as $D_8$, which is a contradiction. Therefore $a\neq r$. Then $x_ax_r\in S_1$, denote $u'=x_ax_{r}$. This implies that $u':v=x_a$, which divides $u:v$, as required. 
\end{proof}

\begin{notation} \label{nota1}
Let $N_D^+(x_\ell)\cap V^+=\{x_t\}$. Define the following sets 
\begin{align*}
S_1' &:= S_1,\\[1mm]
S_2' &:= \{g_{j_1},\ldots,g_{j_{p'}},x_tx_\ell^2,
g_{j_{p'+1}},\ldots,g_{j_q}\}
      = S_2\cup\{x_tx_\ell^2\}, \\ & \qquad \text{where } \deg(g_{j_r})=2
\text{ for every } 1\le r\le p', 
      \\[1mm]
S_3' &:= \{x_rx_\ell^2\in S_3 : r\neq t\}
      = S_3\setminus\{x_tx_\ell^2\},\\[1mm]
S_4' &:= S_4.
\end{align*} 
Note that $HS_1(I_1)+I_2$ is minimally generated by $S_1'\cup S_2'\cup S_3'\cup S_4'$. 
\end{notation}

\begin{lemma}\label{S_2':S_3'}
    Let $D$ be a vertex-splittable weighted oriented graph as in \Cref{s1s2s3s4}. Let $S_2',S_3'$ be the sets as defined in \Cref{nota1}. Assume that $D_6,D_7,D_8$ are not induced subgraphs of $D$. Suppose that $$N_D^+(x_\ell)\cap V^+=\{x_t\}\subseteq N_D(x_1)\text{ and } w_{\ell}=2.$$ 
     For any $u\in S_2'$ and $v\in S_3'$, either $u:v$ is a variable or there exists $u'\in S_1$ such that $u':v$ is a variable which divides $u:v$. 
\end{lemma}
\begin{proof} Let $u\in S_2'$ and $v\in S_3'$. Set $v=x_{r}x_\ell^2\in S_3'$,  where $r\in \{m+1,\ldots,\ell-1\}$ and $r\neq t$. If $u=x_tx_\ell^2$, then $u:v = x_tx_\ell^2:v = x_t$ which is a variable. Now assume that $u \neq x_tx_\ell^2$ and $u:v$ is not a variable. 
Set $u=x_ax_b^{w_b}\in S_2$. Note that one of $x_a$ or $x_b$ is in $N_D(x_1)=\{x_{m+1},\ldots ,x_\ell\}$ because $N_D(x_1)$ is a vertex cover of $D$. 

\noindent
\textbf{Case 1:}
Assume $x_b\in N_D(x_1)$. If $b = \ell$, then $u:v=x_a$, which is a variable, as required. Now onwards we assume that $b\neq \ell$. If $b\neq r$, then $x_bx_{r}\in S_1$, let $u'=x_bx_{r}$, then $u':v=x_b$ which divides $u:v$, as required. Therefore, it remains to consider the case
$b=r$. Thus, $u=x_ax_b^{w_b}$ and $v=x_bx_\ell^2$.
If $a\neq \ell$, then the induced subgraph $D[\{x_1,x_a,x_b,x_\ell\}]$ is same as $D_6$, which is a contradiction. Therefore $a=\ell$. Then we get $u=x_\ell x_b^{w_b}\in S_2$ and $v=x_bx_\ell^2\in S_3'$. If $w_b= 1$, then $u\mid v$, which contradicts that $v=x_bx_\ell^2\in S_3'$. Therefore $w_b \geq 2$,  this gives that $x_b\in N^+_D(x_\ell)\cap V^+$. This implies that $b=t$. This gives that $r=b=t$, which contradicts the fact that $r\neq t$. Thus, either $a \neq \ell$ or $a=\ell$, we get a contradiction. This implies that $b\neq r$. At the beginning of this case, we already shown that the conclusion of the lemma is true if $b\neq r$.  
 
 \noindent
 \textbf{Case 2:} Assume $x_a\in N_D(x_1)$. If $a=\ell$, then by \Cref{subset}, we get $x_b\in N_D(x_1)\setminus\{x_\ell\}$, which falls under Case 1 and we are done in this case. 
Assume $a\neq \ell$. We first show that $a\neq r$. Suppose $a=r$. Then $u=x_ax_b^{w_b}~ \text{and}~ v=x_ax_\ell^2$. Since $u:v$ is not a variable, then $w_b>1$. This implies induced subgraph $D[\{x_1,x_a,x_b,x_\ell\}]$ is same as $D_8$, which is a contradiction. Therefore $a\neq r$. Then $x_ax_r\in S_1$, denote $u'=x_ax_{r}$. This implies that $u':v=x_a$, which divides $u:v$, as required. 
\end{proof}

Now we prove the main result of this section. 

\begin{theorem}\label{vertex splittable} 
Let $D$ be a vertex-splittable weighted oriented graph as in \Cref{s1s2s3s4}. Then $HS_1(I)$ has linear quotient property if and only if $ D_6, D_7,D_8$ are not induced subgraphs of $D$ as in \Cref{fig7}. 
\end{theorem}
\begin{proof}
Suppose $HS_1(I)$ has linear quotient property. Then by  \Cref{D_i not induced}, $D_i$'s are not induced subgraphs of $D$, for $i=6,7,8$. 

Conversely, assume $D_i$'s are not induced subgraphs of $D$, for $i=6,7,8$. To prove $HS_1(I)$ has linear quotient property, we use induction on $\vert V(D)\vert$. If $\vert V(D)\vert=2$, then it is trivial. Assume $\vert V(D)\vert \geq 3$. Since $I_2=I(D\setminus x_1)$ is vertex-splittable, by induction hypothesis, $HS_1(I_2)$ has linear quotients. By \Cref{eq-1} in \Cref{s1s2s3s4}, we have  $HS_1(I)=x_1(HS_1(I_1)+I_2)+HS_1(I_2)$. To show $HS_1(I)$ has linear quotients, by \Cref{sum}, it is enough to prove that $HS_1(I_1)+I_2$ has linear quotients. Now, we show that $HS_1(I_1)+I_2$ has linear quotients. If $w_\ell=1$, then by \Cref{wl-one}, $HS_1(I_1)+I_2$ has linear quotient property. Now onwards, assume $w_\ell \geq 2$.
\vskip 0.2cm 
\noindent 
By \Cref{V+}, we have $| N^+_D(x_{\ell}) \cap V^+| \leq 1$. Now we have two disjoint cases: 
\vskip 0.2cm
\noindent
 {\bf Case 1:} Suppose $N_D^+(x_\ell)\cap V^+=\{x_t\}\nsubseteq N_D(x_1) ~\text{or}~ N_D^+(x_\ell)\cap V^+=\emptyset ~\text{or}~ (N_D^+(x_\ell)\cap V^+=\{x_t\}\subseteq N_D(x_1)$ and $w_{\ell} > 2).$ In this case, show that $S_1\cup \dots \cup S_4$ is a linear quotient order. 
Note that by \Cref{wl-two}, we have that $(S_1\cup S_2)$ has a linear quotient order. It is clear that, for all $u,u'\in S_3$ and $v,v'\in S_4$, the colons
$u:u'$, $v:v'$, and $u:v$ are variables.
By \Cref{s1-S3-S4-colon}, \Cref{colon1} and \Cref{lem-S_2-S_4},  we have 
$$ ((S_1):(S_3)),\qquad ((S_1):(S_4)), \qquad ((S_1\cup S_2):(S_3)), \qquad ((S_2\cup S_3):(S_4)),$$
are generated by variables. Consequently, we get that $S_1\cup S_2\cup S_3\cup S_4$ is a linear quotient order of $HS_1(I_1)+I_2$. 
\vskip 0.2cm

\noindent 
\textbf{Case 2:} Suppose $N_D^+(x_\ell)\cap V^+=\{x_t\}\subseteq N_D(x_1)$ and $w_{\ell}=2$.  Now we will show that $S_1'\cup S_2'\cup S_3'\cup S_4'$ is a linear quotient order of $HS_1(I_1)+I_2$, where $S_i'$ are the sets as defined in \Cref{nota1}.
By \Cref{s1-S3-S4-colon}, \Cref{colon1} and \Cref{lem-S_2-S_4},  we have 
$$ ((S_1'):(S_3')),\qquad ((S_1'):(S_4')), \qquad ((S_1'\cup S_2'):(S_3')), \qquad ((S_2'\cup S_3'):(S_4')).$$ are generated by variables. Thus, to show that $S_1'\cup S_2'\cup S_3'\cup S_4'$ is a linear quotient order for $HS_1(I_1)+I_2$, it is enough to show $$(S_1' \cup S_2')=(u_1,\ldots,u_\alpha,g_{j_1},\ldots,g_{j_{p'}},x_tx_\ell^2,
g_{j_{p'+1}},\ldots,g_{j_q})$$ has a linear quotient order. 

\noindent 
\underline{Claim:} For $u,v\in S_1'\cup S_2'$, either $u:v$ is a variable or there exists $u'\in S_1'$ such that $u':v$ is a variable which divides $u:v$.

 Since $S_1'=S_1$ already has a linear quotient order, the claim holds for $u,v \in S_1'$. If $u\in S_1'$ and $v\in S_2$, then the claim follows from \Cref{wl-two}. If $u,v\in S_2$, then the claim follows from \Cref{wl-two}. If $u\in S_1'$ and $v=x_tx_\ell^2\in S_2'$, then the claim follows from \Cref{s1-S3-S4-colon}. Assume $u=g_{j_l}=x_ax_b\in S_2'$, for $l\in\{1,\ldots,p'\}$ and $v=x_tx_\ell ^{2}\in S_2'$ such that $g_{j_l}:v$ is not variable. Since $u$ is an edge of $D$ and $N_D(x_1)$ is a vertex cover of $D$, then we have one of $x_a$ or $x_b$ is in $N_D(x_1)=\{x_{m+1},\ldots,x_\ell\}$. If $a\neq \ell$ (respt. $b\neq \ell$), then $x_ax_t\in S_1'$ (respt. $x_bx_t\in S_1')$, say $u'=x_ax_t$ (respt. $x_bx_t$ ) such that $u':v=x_a$ (respt. $x_b$), which divides $u:v$, as required.
 
 If $a=\ell$ or $b=\ell$, then by \Cref{subset}, we get either $x_b\in N_D(x_1)\setminus \{x_\ell\}$ or $x_a\in N_D(x_1)\setminus \{x_\ell\}$ respectively. Without loss of generality assume $x_a\in N_D(x_1)\setminus \{x_\ell\}$. Note that $x_t\in N_D(x_1)$. Also $x_t\neq x_\ell$ because $v=x_tx_\ell^2\in S_3$. This implies that $x_tx_a\in S_1$, denote $u'=x_ax_t$ such that $u':v=x_a$, which divides $u:v$, as required. Therefore it remains to prove the claim for $u=x_tx_\ell ^{2}\in S_2'$ and $v=g_{j_l} \in S'_2$,  for $l\in\{p'+1,\ldots,q\}$. 

Let $u=x_tx_\ell ^{2}\in S_2'$ and $v=g_{j_l}=x_cx_d^{w_d}\in S'_2$,  for $l\in\{p'+1,\ldots,q\}$ with $w_d \geq 2$. If $u:v$ is a variable, then nothing to show. Assume $u:v$ is not a variable. If $d\neq \ell$, then $u:v=x_t$, which is a variable. Assume $d\neq \ell$. Now we will consider the following cases: 
\vskip 0.2cm  
\noindent
\textrm{(i)} $d=t~\text{and}~ c=\ell$.
\textrm{(ii)} $d=t~\text{and}~ c\neq \ell$.
\textrm{(iii)} $d\neq t~\text{and}~ c=\ell$.
\textrm{(iv)} $d\neq t~\text{and}~ c\neq \ell$. 

\noindent 
\textrm{(i)}. For $d=t~\text{and}~ c=\ell$, we get $u:v=x_tx_\ell^2:x_\ell x_t^{w_t}=x_\ell$ is a variable, as required.

 \noindent
\textrm{(ii)}. For $d=t~\text{and}~ c\neq \ell$, we get $u=x_tx_\ell^2$ and $v=x_cx_t^{w_t}$. Then the induced subgraph $D[\{x_1,x_t,x_c,x_\ell\}]$ is same as $D_6$, which is a contradiction.

\noindent
  \textrm{(iii)}. For $d\neq t$ and $c=\ell$, we get $u=x_tx_\ell^2$ and $v=x_\ell x_d^{w_d}$. Consequently, $x_t,x_d\in N^+_D(x_\ell)\cap V^+$, contradicting \Cref{V+}.

  \noindent
 \textrm{(iv)}. For $d\neq t$ and $c\neq \ell$, we have $u=x_tx_\ell^2$ and $v=x_cx_d^{w_d}$. Since $v=x_cx_d^{w_d}\in I_2$ is an edge of $D$, then one of $x_c ~\text{or}~x_d ~\text{is in}~ N_D(x_1)$. Without loss of generality, let $x_c\in N_D(x_1)$. Since $c,d\neq \ell$, then $x_c\in N_D(x_1)\setminus \{x_\ell\}$. Now we show that $t\neq c$. Suppose $t=c$. Then $u=x_tx_\ell^2$ and $v=x_tx_d^{w_d}$ is an edge in $D$. This implies that the induced subgraph $D[\{x_1,x_t,x_d,x_\ell\}]$ is same as $D_8$, which is a contradiction. Therefore, $t\neq c$. Then $x_tx_c\in S'_1$, denote $u'=x_tx_c$ such that $u':v=x_t$, which divides $u:v$, as required. This proves the claim and hence the theorem. 
\end{proof}
\begin{example}
   Let $I(D)=x_1I_1+I_2$ be a vertex splittable, where $I_1=(x_3,x_4,x_2^2)$ and $I_2=(x_1x_4,x_2x_3^3)$. By \Cref{betti split}, we have $HS_1(I)=x_1(HS_1(I_1)+I_2)+HS_1(I_2)$. Note that $HS_1(I_1)=(x_3x_4,x_3x_2^2,x_4x_2^2)$ and $I_2=(x_1x_4,x_2x_3^3)$. Now we consider 
   $$ S_1=\{x_3x_4\},~
       S_2'=\{x_1x_4,x_3x_2^2,x_2x_3^3\},
       ~S_3'=\{x_4x_2^2\},
      ~ S_4'=\emptyset.$$
   It is easy to show that $HS_1(I_1)+I_2=(S_1'\cup S_2'\cup S_3' \cup S_4')$ has linear quotients.
\end{example}

\section{Higher homological shift ideals of $I(D)$} \label{sec4} 

In this section, we investigate homological linear quotients of weighted oriented graphs. The main objective of this section is to characterize the radicals of homological shift ideals of weighted oriented graphs in terms of the homological shift ideals of their underlying simple graphs. We also characterize weighted oriented trees whose edge ideals have homological linear quotients. In this section, for any monomial $u={\bf x}^{\bf a}\in R$, we denote $\sqrt{u} := \prod_{a_i > 0} x_i$. 
\vskip 0.2cm
\noindent  
The following lemma will be used in the sequel.
\begin{lemma}\label{supp}
    Suppose that $u,v$ be two monomials in $R$ such that $\sqrt{u}\neq \sqrt{v}$. Then $$(\sqrt{u}:\sqrt{v})~\text{divides}~(u:v).$$
\end{lemma}
\begin{proof}
  Let $u=x_1^{a_1}\cdots x_n^{a_n}$ and $v=x_1^{b_1}\cdots x_n^{b_n}$. Then $${u}=\prod_{i\in \text{supp(u)}} x_i^{\alpha_i}~~ \text{and} ~~{v}=\prod_{i\in\text{supp(v)}} x_i^{\beta_i}, \mbox{ where } \alpha_i \in \{a_1,\ldots,a_r\} \mbox{ and } \beta_i \in \{b_1,\ldots,b_r\}.$$ Therefore, $\displaystyle u:v=\frac{u}{\text{gcd}(u,v)}=\prod_{i\in \text{supp(u)}} x_i^{\alpha_i-\gamma_i}$, ~~where~~~$\displaystyle\text{gcd}(u,v)=\prod _{i\in \text{supp(u)}\cap \text{supp(v)}}x_i^{\gamma_i}$ and $\gamma_i=\text{min}\{{\alpha_i,\beta_i}\}.$ As $\sqrt{u}=\prod_{i\in \text{supp}(u)} x_i$ and $\sqrt{v}=\prod_{i\in \text{supp}(v)} x_i$ then $\sqrt{u}:\sqrt{v}=\frac{\sqrt{u}}{\text{gcd}(\sqrt{u},\sqrt{v})}=\prod _{i\in \text{supp(u)}\setminus \text{supp(v)}}x_i$. This implies that $(\sqrt{u}:\sqrt{v}) \text{ divides } (u:v) $.
\end{proof}

The following proposition gives a sufficient condition for the radical of a monomial ideal to have linear quotients whenever the ideal itself has linear quotients.

\begin{proposition}\label{sqrt}
Let $I=(u_1,\ldots,u_m)$ be a monomial ideal in $R$. Suppose $I$ has linear quotient property. 
Assume $\sqrt{u_i}, \sqrt{u_j} \in \mathcal{G}(\sqrt{I})$, for all $i,j$ with $\sqrt{u_i}\neq \sqrt{u_j}$. Then $\sqrt{I}$ has linear quotient property.
\end{proposition}
\begin{proof}
Suppose $u_1 < \dots < u_m$ is a linear quotient order of $I$. Let 
\begin{align*}
i_1 &:= \min\{\, j \mid \sqrt{u_j} = \sqrt{u_1} \,\} = 1, \\
i_2 &:= \min\{\, j \mid \sqrt{u_j} \neq \sqrt{u_{i_1}} \,\}, \\
&\ \ \vdots \\
i_s &:= \min\{\, j \mid \sqrt{u_j} \neq \sqrt{u_{i_1}}, \ldots, \sqrt{u_{i_{s-1}}} \,\}.
\end{align*}
 So $1=i_1 < \dots < i_s\leq m$. Since the radicals $\sqrt{u_{i_1}},\ldots,\sqrt{u_{i_s}}$ are pairwise distinct, the assumption ensures that none of them divides another. Hence,
$$
\sqrt{I}=(\sqrt{u_{i_1}},\ldots,\sqrt{u_{i_s}})
$$
is the minimal generating set of $\sqrt{I}$.  
\vskip 0.2cm
\noindent 
{\bf Claim:} $\sqrt{u_{i_1}}, \ldots, \sqrt{u_{i_s}}$ is a linear quotient order of $\sqrt{I}$.\\  
Let $i_j < i_k$. If $\sqrt{u_{i_j}}: \sqrt{u_{i_k}}$ is a variable, then noting to show. Assume  $\sqrt{u_{i_j}}: \sqrt{u_{i_k}}$ is not a variable. By \Cref{supp}, $u_{i_j}:u_{i_k}$ is also not a variable. Since $I$ has linear quotient property, then there exists $l<i_k$ such that $u_l:u_{i_k}=x$, a variable, such that $x \mid u_{i_j}:u_{i_k}$. Then $x \in \supp(u_{i_j})$. Since $i_k=\min\{j :\sqrt{u_j} \neq \sqrt{u_{i_1}},\ldots,\sqrt{u_{i_{k-1}}}\}$ then $\sqrt{u_l} \neq \sqrt{u_{i_k}} $, otherwise $\sqrt{u_l} = \sqrt{u_{i_k}} $ contradicts the minimality of $i_k$. Since $u_l:u_{i_k}=x$,  then by \Cref{supp}, we get $\sqrt{u_l}:\sqrt{u_{i_k}}=x$. This gives that $x \not \in \supp(\sqrt{u_{i_k}})$. Then we have $x \in \supp(\sqrt{u_{i_j}})$ and $x \not \in \supp(\sqrt{u_{i_k}})$. Therefore, $x \mid \sqrt{u_{i_j}}:\sqrt{u_{i_k}}$. This implies that $\sqrt{u_l}:\sqrt{u_{i_k}}=x\mid \sqrt{u_{i_j}}:\sqrt{u_{i_k}}$. Furthermore, we have $\sqrt{u_l}=\sqrt{u_{i_{\ell}}}$, for some $i_{\ell} < i_k$. This proves the claim and hence the proposition. 
\end{proof}

In \Cref{sqrt}, we can not drop the assumption that $\sqrt{u_i}, \sqrt{u_j} \in \mathcal{G}(\sqrt{I})$, for all $i,j$ with $\sqrt{u_i}\neq \sqrt{u_j}$. See the following example. 
\begin{example}
     Consider the $5$-cycle $\mathcal{C}_5$, whose edge ideal $$I=I(\mathcal{C}_5)=(x_1x_2,x_2x_3,x_3x_4,x_4x_5,x_1x_5).$$ Then $I^2$ has linear quotients, whereas $I=\sqrt{I^2}$ does not. 
\end{example}

The following corollary is an immediate consequence of \Cref{sqrt}.
\begin{corollary} \label{cor1}
    Let $I$ be a monomial ideal such that the support of all minimal generators of $I$ have the same size. If $I$ has linear quotients, then $\sqrt{I}$ has linear quotients.   
\end{corollary}
\begin{proof}
    Note that if $I$ is monomial ideal such that the support of all minimal generators of $I$ have the same size, then the assumption of \Cref{sqrt} is satisfied.  
\end{proof}

\begin{corollary}\label{I(G)}
  Let $D$ be a weighted oriented graph with the underlying simple graph $G$. Suppose $I(D)$ has linear quotient property. Then $I(G)$ has linear quotient property.    
\end{corollary}
\begin{proof}
     Since the support of all minimal generators of $I(D)$ have the same size $2$, the result follows from \Cref{cor1}. In fact, if $I(D)$ is minimally generated by $u_1,\ldots,u_m$, then $\sqrt{I(D)}$ is minimally generated by $\sqrt{u_1},\ldots,\sqrt{u_m}$. 
\end{proof}

We are now ready to prove the main result of this section.
\begin{theorem}\label{T-radicalhomological}
Let $D$ be a weighted oriented graph with the underlying simple graph $G$ such that $I(D)$ has linear quotient property. Then 
$$\sqrt{HS_k(I(D))}=HS_k(I(G)) \;\; \text{ for all } k\geq 0.$$    
\end{theorem}
\begin{proof}
Let $D$ be a weighted oriented graph such that $I=I(D)$ has linear quotient property. By \Cref{corollary-2..5}, we have $D_i$ are not induced subgraphs of $D$ for $i\in \{1,\ldots,5\}$. Furthermore, by \Cref{L-increasing-admissible}, $I$ has a degree-increasing linear quotients order. Let $u_1,\ldots,u_m$ be a degree-increasing linear quotients order of $I$. Then by \cref{L-homologicalshift}, we have $$HS_k(I)=(u_iX_A~|~i=1,\ldots,m, ~A\subseteq \text{set}_I (u_i),~ |A|=k).$$
Note that the support of all minimal generators of $I$ have the same size. Therefore by \Cref{cor1}, $\sqrt{I}$ has a linear quotient order $\sqrt{u_1},\ldots,\sqrt{u_m}$. 
\vskip 0.2cm
\noindent 
\textbf{Claim:} $\text{set}_I (u_i)=\text{set}_{\sqrt{I}}(\sqrt{u_i})$, for all $i=1,\ldots,m$.
\vskip 0.2cm
Let $a\in \text{set}_{\sqrt{I}}(\sqrt{u_i})$ then there exists $r<i$ such that $\sqrt{u_r}:\sqrt{u_i}=x_a$. This implies that $u_r=x_ax_b^{w_b}~\text{or}~ x_bx_a^{w_a}$ and $u_i=x_bx_j^{w_j}~\text{or}~ x_jx_b^{w_b}$.

\noindent
\textbf{Case 1:} Let $u_r=x_ax_b^{w_b}$ and $u_i=x_bx_j^{w_j}$. If $w_b=1$ then $u_r:u_i=x_a$. Suppose $w_b>1$. Since $u_r< u_i$ in the degree-increasing linear quotient order of $I$, then $w_b\leq w_j$. Since $I(D)$ has linear quotient property, then $D_3$ can not be an induced subgraph of $D$. Therefore,  $(x_a,x_b)\in E(D)$ and $(x_b,x_j)\in E(D)$ gives $(x_j,x_a)\in E(D)$. This implies that $x_jx_a\in \mathcal{G}(I(D))$. Since $\text{deg}(x_jx_a)=2$, then there exist a $k<i$ such that $u_k=x_jx_a$ and  $u_k:u_i=x_a$. Thus, in either case, we get $a\in \text{set}_I(u_i)$, as required.
\vskip 0.2cm
\noindent
\textbf{Case 2:} Let $u_r=x_ax_b^{w_b}$ and $u_i=x_jx_b^{w_b}$. Then $u_r:u_i=x_a$ implies that $a\in \text{set}_I(u_i)$.
\vskip 0.2cm
\noindent
\textbf{Case 3:} Let $u_r=x_bx_a^{w_a}$ and $u_i=x_bx_j^{w_j}$. Now we show that $w_a=1$. Suppose $w_a>1$. Since $u_r< u_i$ in the degree-increasing linear quotient order of $I$, then $w_a\leq w_j$. Then $x_a,~x_j\in N_D^+(x_b)\cap V^+$, contradicting \Cref{V+}, $|N_D^+(x_b)\cap V^+|\leq 1$. Therefore $w_a=1$. This implies that $u_r:u_i=x_a$. Hence we get $a\in \text{set}_I(u_i)$. 
\vskip 0.2cm
\noindent
\textbf{Case 4:} Let $u_r=x_bx_a^{w_a}$ and $u_i=x_jx_b^{w_b}$. If $w_a=1$ then $u_r:u_i=x_a$. Suppose $w_a>1$. Since $u_r< u_i$ in the degree-increasing linear quotient order of $I$, then $w_b\leq w_j$. Since $I(D)$ has linear quotient property, then $D_3$ can not be an induced subgraph of $D$. Therefore,  $(x_b,x_a)\in E(D)$ and $(x_j,x_b)\in E(D)$ gives $(x_j,x_a)\in E(D)$. This implies that $x_jx_a\in \mathcal{G}(I(D))$. Since $\text{deg}(x_jx_a)=2$, then there exist a $k<i$ such that $u_k=x_jx_a$ and  $u_k:u_i=x_a$. Thus, in either case, we get $a\in \text{set}_I(u_i)$, as required.

Conversely let $a\in \text{set}_I (u_i)$ then there exists $r<i$ such that $u_r:u_i=x_a$. Since $u_r$ and $u_i$ are edges of $D$, then $\sqrt{u_r}\neq \sqrt{u_i}$. Then by Lemma \ref{supp}, we get $\sqrt{u_r}:\sqrt{u_i}=x_a$. This implies that $a\in \text{set}_{\sqrt{I}} (\sqrt{u_i})$.

Therefore, $\text{set}_I (u_i)=\text{set}_{\sqrt{I}} (\sqrt{u_i})$, for all $i=1,\ldots,m$. Thus, 
\begin{eqnarray*}
    \sqrt{HS_k(I(D))} &=& \sqrt{(u_iX_A~|~i=1,\ldots,m, ~A\subseteq \text{set}_I (u_i),~ |A|=k)} \\
    &=& (\sqrt{u_i}X_A~|~i=1,\ldots,m, ~A\subseteq \text{set}_I (u_i),~ |A|=k) \\
    &=& (\sqrt{u_i}X_A~|~i=1,\ldots,m, ~A\subseteq \text{set}_{\sqrt{I}} (\sqrt{u_i}),~ |A|=k )  \\
    &=& HS_k(I(G)). 
\end{eqnarray*}
\end{proof}

In Theorem \ref{T-radicalhomological}, we can not drop the assumption that $I(D)$ has linear quotients; see the below example. 

\begin{example}
Let $I(D)=(x_1x_2^3,x_2x_3^3,x_3x_4^3,x_5x_4^3)$. Then $I(D)$ does not have linear quotients. Furthermore, $\sqrt{HS_2(I(D))}=(x_1x_2x_3x_4)$ and $HS_2(I(G))=(x_1x_2x_3x_4x_5)$. Hence  $\sqrt{HS_2(I(D))}\neq HS_2(I(G))$.    
\end{example}

The following corollary shows that the homological linear quotient property is inherited by the underlying simple graph.
\begin{corollary}\label{HS_K(G)}
Let $D$ be a weighted oriented graph such that $I(D)$ has homological linear quotients. Then $I(G)$ has homological linear quotients.     
\end{corollary}
\begin{proof}
  Let $I(D)$ has homological linear quotients. In particular, $I(D)$ has linear quotients. Then by \Cref{T-radicalhomological}, we have $$\sqrt{HS_k(I(D))}=HS_k(I(G)) ~\text{for all}~ k\geq 0.$$ Therefore, by \Cref{cor1}, we get that $I(G)$ has homological linear quotients.   
\end{proof}

The following corollary gives necessary conditions for the edge ideal of a weighted oriented graph to have homological linear quotients.
\begin{corollary}\label{complement}
    Let $D$ be a weighted oriented graph and $I(D)$ has homological linear quotients. Then $D$ is $D_i$-free for any $i=1,\ldots,8$. Moreover, the underlying graph $G$ is co-chordal and $H_n^c$-free for any $n\geq 6$.
\end{corollary}
\begin{proof}
    Suppose $I(D)$ has homological linear quotients. In particular, $I(D)$ has the linear quotient property. Hence by \Cref{corollary-2..5} and \Cref{D_i not induced}, we have $D_i$ can not be induced subgraph of $D$ for any $i=1,\ldots,8$. Moreover, by \Cref{HS_K(G)}, we have $I(G)$ has homological linear quotients. By, \Cref{Hs_k(I(G))}, we have  the underlying graph $G$ of $D$ is co-chordal and $H_n^c$-free for any $n\geq 6$. 
\end{proof}

The following corollary shows that several important algebraic properties of homological shift ideals are inherited by the underlying simple graph.
\begin{corollary} \label{sec4cor1}
    Let $D$ be a weighted oriented graph and $G$, its underlying simple graph. Suppose  $I=I(D)$ has linear quotients. Let $k\geq 0$.
Suppose that $HS_k(I(D))$ satisfies one of the following properties:
\begin{enumerate}
    \item Cohen--Macaulay,
    \item Gorenstein,
    \item sequentially Cohen--Macaulay,
    \item generalized Cohen--Macaulay,
    \item Buchsbaum.
\end{enumerate}
Then $HS_k(I(G))$ satisfies the corresponding property.
\end{corollary}
\begin{proof}
    Follows from \Cref{T-radicalhomological} and \cite[Theorem 2.6]{htt05}. 
\end{proof}

The below example conveys that the converse of \Cref{sec4cor1} is need not be true.
\begin{example}
    Let $I(D)=(x_1x_3,x_1x_2^2,x_2x_3^3)$ be the edge ideal of a $3$-cycle such that $I(D)$ has linear quotients. Furthermore, $HS_1(I(G))=(x_1x_2x_3)$ and $HS_1(I(D))=(x_1x_2x_3^2,x_1x_2^2x_3)=$. Note that $HS_1(I(G))$ is Cohen-Macaulay whereas $HS_1(I(D))$ is not. 
\end{example}

To prove the characterization of weighted oriented trees with homological linear quotients, we first establish the following lemma and propositions.

\begin{lemma}\label{tree}
    Let $I =(x_1,\ldots,x_{s-1},x_s^{w_s}) \subseteq R$ be an ideal and $w_s\geq 1$. Then $HS_k(I)$ is a weakly polymatroidal ideal, for any $k\geq 0$. Furthermore, $I$ has homological linear quotients. 
    \end{lemma}
\begin{proof} Fix $k\geq 0$. Let $I=(x_1,\ldots,x_{n-1},x_s^{w_s})$. Then we have $I=HS_0(I)$ has linear quotients. Assume $i\geq 1$. Note that 
    $\text{set}_I(x_i)=\{1,\ldots,{i-1}\}$ for all $i\in\{1,\ldots,s-1\}$ and $\text{set}_I(x_s^{w_s})=\{1,\ldots,{s-1}\}$. Then by \Cref{L-homologicalshift},  $HS_k(I)=HS_k(x_1,\ldots,x_{s-1})+A$, where 
  \begin{align*}
A &= (x_{i_1}\ldots x_{i_k}x_s^{w_s}\mid 1\leq i_1<\cdots <i_k\leq s-1) \text{ and } \\
HS_k(x_1,\ldots,x_{s-1}) 
&= x_{k+1}(x_1\ldots x_k) 
+ x_{k+2}(x_{i_1}\ldots x_{i_k}\mid 1\leq i_1<\cdots <i_k\leq k+1) \\
&\quad + \dots + x_{s-1}(x_{i_1}\ldots x_{i_k}\mid 1\leq i_1<\cdots <i_k\leq s-2)\\
&=(x_{i_1}\ldots x_{i_{k+1}}\mid 1\leq i_1<\cdots <i_{k+1}\leq s-1)
\end{align*}
    Note that $HS_k(x_1,\ldots,x_{s-1})$ is the square-free Veronese ideal of degree $k+1$ in the variables $x_1,\ldots,x_{s-1}$. Thus, by \cite[Theorem 4.2]{hh06}, it is a weakly polymatroidal ideal. It is easy to see that $A$ is a weakly polymatroidal ideal. Now, we will show that $HS_k(x_1,\ldots,x_{s-1})+A$ is weakly polymatroidal. Let  $u=x_{p_1}\cdots x_{p_k}x_s^{w_s}\in A$ and $v=x_{q_1}\cdots x_{q_k}x_{q_{k+1}}\in HS_k(x_1,\ldots,x_{s-1})$, where $1\leq p_1<\cdots <p_k\leq s-1$ and $1\leq q_1<\cdots <q_k < q_{k+1}\leq s-1$. The following cases arise: 
    \vskip 0.2cm
    \noindent 
   {\bf Case 1:}  Suppose $u\underset{lex}{>}v$. Then there exists an index $r$ such that  $p_1=q_1,\ldots,p_{r-1}=q_{r-1}$ and $p_r<q_r$. Now we show that there exists $x_j \underset{lex}{<} x_{p_r}$ such that $\frac{v}{x_j}x_{p_r} \in HS_k(I)$. If $r>k$, then $u=x_{p_1}\cdots x_{p_k}x_s^{w_s}$ and $v=x_{p_1}\cdots x_{p_k}x_{q_{k+1}}$, which implies that $u\underset{lex}{<}v$, a contradiction. 
    If $r=k$, then $p_k<q_k$ and $p_i=q_i$ for all $i\in \{1,\ldots,k-1\}$. This implies that  $x_{q_{k+1}} \underset{lex}{<} x_{p_r}$ and $\frac{v}{x_{q_{k+1}}}x_{p_r} =x_{q_1} \cdots x_{q_{k}}x_{p_r} \in HS_k(x_1,\ldots,x_{s-1}) \subseteq HS_k(I)$, as required. Now, assume $r < k$. That is, $q_r < q_k$. This implies that $x_{q_k} \underset{lex}{<} x_{q_r}$. Then $\frac{v}{x_{q_k}}x_{p_r}=x_{p_1}\cdots x_{p_r} x_{q_r}\cdots x_{q_{k-1}}x_{q_k}\in HS_k(x_1,\ldots,x_{s-1})$. 
    \vskip 0.2cm 
    \noindent
    {\bf Case 2:} Suppose $u\underset{lex}{<}v$. Then there exists an index $r\in \{1,\ldots,k+1\}$ such that  $p_1=q_1,\ldots,p_{r-1}=q_{r-1}$ and $q_r<p_r$. Now we show that there exists $x_j \underset{lex}{<} x_{q_r}$ such that $\frac{u}{x_j}x_{q_r} \in HS_k(I)$. Suppose $r=k+1$. Since $x_{q_r}=x_{q_{k+1}} \underset{lex}{>} x_s$, then $\frac{u}{x_n}x_{q_{k+1}}=x_{q_1}\cdots x_{q_k}x_{q_{k+1}}x_s^{w_s-1}\in HS_k(I)$, as required. If $r\neq k+1$ i.e., $r<k$, then $q_r < p_r \leq p_k$. This implies that $x_{p_k}\underset{lex}{\geq}x_{p_r}\underset{lex}{>} x_{q_r}$ and $\frac{u}{x_{p_k}}x_{q_r}\in A$, as required. 
    \vskip 0.2cm 
    \noindent 
    Thus $HS_k(I)$ is a weakly polymatroidal ideal. Consequently, $HS_k(I)$ has linear quotients. Therefore, $I$ has homological linear quotients.
\end{proof}

\begin{proposition}\label{tree1}
     Let $I=x_1^{w_1}(x_2,\ldots, x_m)+x_1(x_{m+1},\ldots,x_s) \subset R=\mathbb{K}[x_1,\ldots,x_n]$ be an ideal. Then $I$ has homological linear quotients.
\end{proposition}

\begin{proof}
  Let $I=x_1^{w_1}(x_2,\ldots, x_m)+x_1(x_{m+1},\ldots,x_s)$. To prove $HS_k(I)$ has linear quotients for all $k\geq 0$, we use induction on $s$. Note that $I$ is vertex-splittable with the vertex splitting $I=x_{m+1}I_1+I_2$, where $I_1=(x_1)$ and $I_2=x_1(x_{m+2},\ldots,x_s)+x_1^{w_1}(x_2,\ldots,x_m)$. Thus $HS_0(I)=I$ has linear quotients. Assume $k\geq 1$. By \Cref{betti split}, we get $HS_k(I)=x_{m+1}(HS_k(I_1)+HS_{k-1}(I_2))+HS_k(I_2)=x_{m+1}HS_{k-1}(I_2)+HS_k(I_2)$, because $HS_k(I_1)=0$. By induction $HS_k(I_2)$ has linear quotients for all $k\geq 0$.  Since $HS_k(I_2)\subset HS_{k-1}(I_2)$, then by \Cref{sum} we get $HS_k(I)$ has linear quotients. Therefore, $I$ has homological linear quotients.
\end{proof}

The following lemma provides a necessary and sufficient condition for a weighted oriented star graph to have homological linear quotients.
\begin{proposition}\label{star}
 Let $D$ be a weighted oriented star graph with the underlying simple graph $G$ as in \Cref{fig9}. Then $HS_k(I(D))$ has linear quotients for all $k\geq 0$ $\iff$ $D_1$ and $D_2$ are not induced subgraphs of $D$.
\end{proposition}

\begin{proof}
Let  $HS_k(I(D))$ has linear quotients for all $k\geq 0$. In particular, $I(D)$ has linear quotients. Then by \Cref{corollary-2..5}, we have $D_1,D_2$ are not induced subgraphs of $D$.

Conversely let, $D$ be a star graph with center $x_1 \in V(D)$ which is a sink. Then  $I(D)=x_1^{w_1}(x_2,\ldots, x_n)$. Then by \cite[Proposition 1.7]{hmmz21}, we get $HS_k(I)=HS_0(x_1^{w_1})HS_k(J)=x_1^{w_1}HS_k(J)$, where $J=(x_2,\ldots,x_n)$. By \cref{tree}, we have $HS_k(J)$ has linear quotients. This gives that $HS_k(I)$ has linear quotients. 

Let $x_1 \in V(D)$ be a non-sink vertex. Then $(x_1,x_i)\in E(D)$, for some $i\in\{2,\ldots,n\}$. Assume $w_i>1$. Then $x_i$ is not a source vertex. This gives that $(x_1,x_i)\in E(D)$. Now onwards we have $(x_1,x_i)\in E(D)$. Then $w_j=1$ for all $j\in\{2,\ldots,n\}\setminus \{i\}$, otherwise  the induced subgraph $N_D[\{x_1,x_i,x_j\}]$ is same as $D_2$, which is contradiction. If $w_1>1$, then $x_1$ is not a source vertex. This gives that $(x_j,x_1)\in E(D)$. Then the induced subgraph $N_D[\{x_1,x_i,x_j\}]$ is same as $D_1$, which is contradiction. Therefore $w_1=1$. Thus we have $I(D)=x_1(x_1,\ldots,x_i^{w_i},\ldots,x_n)$, which is vertex splitting of $I(D)$. By \Cref{betti split}, we have $HS_k(I(D))=x_1HS_k((x_1,\ldots,x_i^{w_i},\ldots,x_n))$. Then by \Cref{tree}, we have $HS_k((x_1,\ldots,x_i^{w_i},\ldots,x_n))$ has linear quotients for all $k\geq 0$. Therefore $HS_k(I(D))$  has linear quotients for all $k\geq 0$.

If $w_i=1$. Then $I(D)=x_iI_1+I_2$ is a vertex-splitting of $I(D)$, where $I_1=(x_1)$ and $I_2=x_1^{w_1}(x_l\mid x_l \in N_D^-(x_1))+x_1(x_l\mid x_l \in N_D^+(x_1))$.  By \Cref{betti split}, we have $$HS_k(I(D))=x_2HS_k(I_1)+x_2HS_{k-1}(I_2)+HS_k(I_2).$$  By \Cref{vertex splittable}, we have $HS_1(I(D))$ has linear quotients. By \Cref{tree1}, we get $HS_k(I_2)$ has linear quotients for all $ k\geq 0$. Note that $HS_k(I_1)=0$ for all $k\geq 2$. Therefore, $HS_k(I)=x_2HS_{k-1}(I_2)+HS_k(I_2)$ for all $k\geq 2$. Then by \Cref{sum}, $HS_k(I)$ has linear quotients for all $k\geq2$.
\end{proof}

We are now ready to prove the following characterization of weighted oriented trees with homological linear quotients.

\begin{figure}[htbp]
\centering
\resizebox{0.85\textwidth}{!}{%
\begin{tikzpicture}[
    vertex/.style={circle,draw,fill=white,minimum size=3pt,inner sep=0pt},
    dot/.style={circle,fill=black,minimum size=1.2pt,inner sep=0pt},
    every node/.style={font=\Large},
    thick
]
\def\r{2.2}
\node[vertex] (v1) at (0,0) {};
\node[vertex,label=left:$x_2$]        (v2) at (180:\r) {};
\node[vertex,label=below left:$x_3$]  (v3) at (225:\r) {};
\node[vertex,label=below:$x_4$]       (v4) at (270:\r) {};
\node[vertex,label=below right:$x_5$] (v5) at (315:\r) {};
\node[vertex,label=right:$x_6$]       (v6) at (360:\r) {};
\node[vertex,label=below right:$x_7$] (v7) at (45:\r) {};
\node[vertex,label=above right:$x_8$] (v8) at (90:\r) {};
\node[vertex,label=above:$x_n$]       (vn) at (135:\r) {};
\foreach \x in {v2,v3,v4,v5,v6,v7,v8,vn}
    \draw (v1)--(\x);
\node[fill=white,inner sep=1pt] at ($(v1)+(22.5:1.15)$) {$x_1$};
\foreach \a in {102,112,122}
    \node[dot] at ($(v1)+(\a:0.95*\r)$) {};
\begin{scope}[xshift=9cm]
\node[vertex,label=below:$x_1$] (a1) at (0,0) {};
\node[vertex,label=below:$x_2$] (a2) at (1.8,0) {};
\node[vertex,label=below left:$x_3$] (a3) at (3.8,0) {};
\draw (a1)--(a2)--(a3);
\node[vertex,label=below:$x_4$]        (b4) at ($(a3)+(290:2.4)$) {};
\node[vertex,label=below right:$x_5$]  (b5) at ($(a3)+(325:2.4)$) {};
\node[vertex,label=right:$x_6$]        (b6) at ($(a3)+(350:2.4)$) {};
\node[vertex,label=right:$x_7$]        (b7) at ($(a3)+(15:2.4)$) {};
\node[vertex,label=above right:$x_8$]  (b8) at ($(a3)+(40:2.4)$) {};
\node[vertex,label=above:$x_n$]        (bn) at ($(a3)+(70:2.4)$) {};
\foreach \x in {b4,b5,b6,b7,b8,bn}
    \draw (a3)--(\x);
\foreach \a in {48,55,62}
    \node[dot] at ($(a3)+(\a:2.3)$) {};
\end{scope}
\end{tikzpicture}%
}
\caption{Star graph and broom graph}
\label{fig9}
\end{figure}
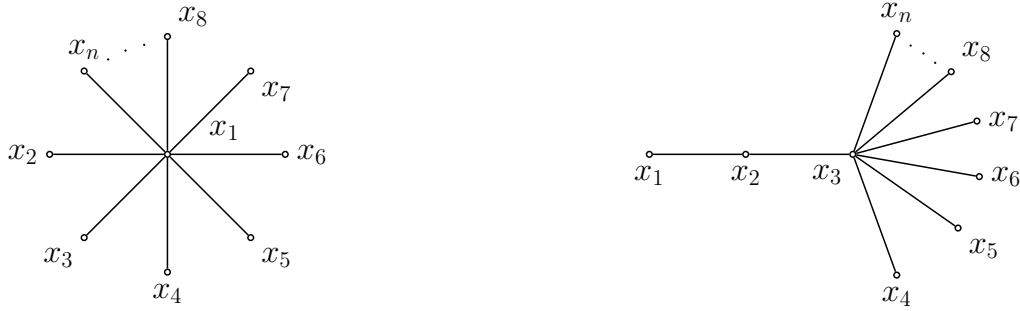

\begin{theorem}\label{thm:hmologicalshift:tree}
  Let $D$ be a weighted oriented tree with its underlying simple graph $G$. Then 
  $HS_k(I(D))$ has linear quotients for all $k\geq 0$ \;$\iff$  \;$G$ is a star graph or a broom graph as in \Cref{fig9} and $D_i$ are not induced subgraphs of $D$ for all $i\in \{1,2,5,6,8\}$.
\end{theorem}

\begin{proof}
Suppose $HS_k(I(D))$ has linear quotient property for all $k\geq0$. Then by \cref{complement}, we get that $G$ is a tree with $H_6^c$-free and co-chordal. Since $G$ is co-chordal, then $G$ can not have a path of length $ \geq 4$. This leads to that $G$ is same as \Cref{fig9}. 
\vskip 0.2cm 
\noindent 
Conversely, suppose $D$ is a weighted oriented tree such that $D_i$ are not induced subgraphs of $D$ for all $i\in \{1,2,5,6\}$ and $G$ is as in \Cref{fig9}. We will show that $HS_k(I(D))$ has linear quotients for all $k\geq 0$. If $G$ is a star graph, then by \Cref{star}, we have $I(D)$ has homological linear quotients, as required. Assume  $G$ is the broom graph as in \Cref{fig9}. We divide the proof in the following cases:
\vskip 0.2cm    
\noindent  
\textbf{Case 1:} Suppose $w_i=1$ for all $i$. That is, $I(D)=I(G)$. In this case, we can write $I(D)=x_3(x_2,x_4,\ldots,x_n)+(x_2x_1)$ which is a vertex splitting of $I(D)$. By \Cref{vertex splittable}, we have $HS_1(I(D))$ has linear quotient property. Assume $k\geq 2$. By \Cref{betti split}, we have
\begin{align*}
    HS_k(I(D))=x_3 HS_k((x_2,x_4,\ldots,x_n))+x_3HS_{k-1}((x_2x_1))+HS_k((x_2x_1)),~\text{for all}~ k\geq 2.
\end{align*}
Since  $HS_k((x_2x_1))=0$ for all $k \geq 2$, this implies that  $HS_k(I(D))=x_3HS_k((x_2,x_4,\ldots,x_n))$. By \Cref{tree}, $HS_k((x_2,x_4,\ldots,x_n))$ has linear quotients for all $k\geq 2$. Therefore $HS_k(I(D)$ has linear quotients for all $k\geq 0$. 
\vskip 0.2cm    
\noindent  
\textbf{Case 2:} Suppose $w_1>1$, then $x_1$ is not a source vertex. This implies that $(x_2,x_1)\in E(D)$. If $w_2>1$, then $(x_3,x_2)\in E(D)$ and the induced subgraph $D[\{x_1,x_2,x_3\}]$ is same as $D_1$, which is a contradiction. Therefore $w_2=1$. If $w_3>1$, then $x_3$ is not a source vertex. This implies that $(x_i,x_3)\in E(D)$, for some $i\in \{2, 4,\ldots,n\}$. If $i=2$, then the induced subgraph $D[\{x_1,x_2,x_3\}]$ is same as $D_2$, which is a contradiction. If $i\neq 2$, then the induced subgraph $D[\{x_1,x_2,x_3,x_i\}]$ is same as $D_6$, which is a contradiction. Therefore, $w_3=1$. Assume $w_j>1$ for some $j\in \{4,\ldots,n\}$, then $x_j$ is not a source vertex. This implies that $(x_3,x_j)\in E(D)$ and the induced subgraph $D[\{x_1,x_2,x_3,x_j\}]$ is same as $D_8$, which is a contradiction. This implies that $w_j=1$, for all $j\in\{4,\ldots,n\}$. 

Thus we have $I(D)=x_3(x_2,x_4,\ldots,x_n)+(x_2x_1^{w_1})$ which is a vertex splitting of $I(D)$. By \Cref{betti split}, we have
    $$HS_k(I(D))=x_3HS_k((x_2,x_4,\ldots,x_n))+x_3HS_{k-1}((x_2x_1^{w_1}))+HS_k((x_2x_1^{w_1})).$$ 
By \Cref{vertex splittable}, we have $HS_1(I(D))$ has linear quotients. Since  $HS_k((x_2x_1^{w_1}))=0$ for all $k>1$, then $HS_k(I(D))=x_3HS_k((x_2,x_4,\ldots,x_n))$. By \Cref{tree}, $HS_k((x_2,x_4,\ldots,x_n))$ has linear quotients. Therefore $HS_k(I(D)$ has linear quotients for all $k\geq 0$.
\vskip 0.2cm   
\noindent
\textbf{Case 3:} Suppose $w_2>1$. Then $x_2$ is not a source vertex. This gives that  $(x_1,x_2)\in E(D)$ or $(x_3,x_2)\in E(D)$. Assume $w_1>1$. Then $x_1$ is not a source vertex. If $(x_1,x_2)\in E(D)$, then we get that $x_1$ is source vertex, which is a contradiction. If $(x_3,x_2)\in E(D)$ then the induced subgraph $D[\{x_1,x_2,x_3\}]$ is same as $D_1$, which is a contradiction. Hence $w_1=1$. Assume $w_3>1$. Then we derive a contradiction. Note that $x_3$ is not a source vertex. This implies that $(x_i,x_3)\in E(D)$, for some $i\in \{2, 4,\ldots,n\}$.  Suppose $(x_1,x_2)\in E(D)$. If $i=2$, then the induced subgraph $D[\{x_1,x_2,x_3\}]$ is same as $D_1$, which is a contradiction. If $i\neq 2$, then the induced subgraph $D[\{x_1,x_2,x_3,x_i\}]$ contains $D_1$ or $D_2$, which is a contradiction. Suppose $(x_3,x_2)\in E(D)$. Then the induced subgraph  $D[\{x_2,x_3,x_i\}]$ is same as $D_1$, which is a contradiction. Therefore, $w_3=1$.

Now we will show that $w_i=1$, for all $i\in \{4,\ldots,n\}$. Suppose $w_i>1$, for some $i\in \{4,\ldots,n\}$. Then $x_i$ is not a source vertex. This gives that $(x_3,x_i)\in E(D)$. Since $w_2>1$, we have  either $(x_1,x_2)\in E(D)$ or $(x_3,x_2)\in E(D)$. If $(x_1,x_2)\in E(D)$ then the induced subgraph $D[\{x_1,x_2,x_3,x_i\}]$ is same as $D_6$, which is a contradiction. If $(x_3,x_2)\in E(D)$, then the induced subgraph $D[\{x_2,x_3,x_i\}]$ is same as $D_2$, which is a contradiction. Therefore $w_i=1$ for all $i\in \{4,\ldots,n\}$. Therefore,
\[
I(D)=
\begin{cases}
x_3(x_2^{w_2},x_4,\ldots,x_n)+(x_1x_2^{w_2}), & \text{if } (x_1,x_2)\in E(D) ~\text{and}~ (x_3,x_2)\in E(D),\\[1mm]
x_3(x_2,x_4,\ldots,x_n)+(x_1x_2^{w_2}), & \text{if } (x_1,x_2)\in E(D) ~\text{and}~ (x_2,x_3)\in E(D),\\[1mm]
x_3(x_2^{w_2},x_4,\ldots,x_n)+(x_1x_2), & \text{if } (x_3,x_2)\in E(D) ~\text{and}~ (x_2,x_1)\in E(D).
\end{cases}
\]
Each of these decompositions is a vertex splitting of $I(D)$, say $I(D)=x_3 I_1+I_2$. Then by \Cref{betti split}, we have $HS_k(I(D))=x_3HS_k(I_1)+x_3HS_{k-1}(I_2)+HS_k(I_2)$. By \Cref{vertex splittable}, we have $HS_1(I(D))$ has linear quotients. Since  $HS_k(I_2)=0$ for all $k\geq 2$, then we get $HS_k(I(D))=x_3HS_k(I_1)$ for all $k\geq 2$. By \Cref{tree}, we have $HS_k(I_1)$ has linear quotients. Therefore $HS_k(I(D))$ has linear quotients for all $k\geq 2$.
\vskip 0.2cm 
\noindent
\textbf{Case 4:} Suppose $w_3>1$. Then $x_3$ is not a source vertex. This gives that  $(x_i,x_3)\in E(D)$ for some $i\in \{2,4,\ldots,n\}$. Now, we will have the following subcases:
\vskip 0.2cm
\noindent
\textbf{Subcase 4(a):}  Suppose $i=2$, i.e., $(x_2,x_3)\in E(D)$. If $w_1>1$, then the induced subgraph $D[\{x_1,x_2,x_3\}]$ is same as $D_2$, which is a contradiction. Therefore $w_1=1$. If $w_2>1$. Then $x_2$ is not a source vertex. Then $(x_1,x_2)\in E(D)$, which gives that the induced subgraph $D[\{x_1,x_2,x_3\}]$ is same as $D_2$, which is a contradiction. Therefore $w_2=1$. If $w_j>1$, for some $j\in\{4,\ldots,n\}$, then the induced subgraph $D[\{x_2,x_3,x_j\}]$ is same as $D_1$, which is a contradiction. Thus, $w_j=1$ for all $j\in \{4,\ldots,n\}$. Consequently, $(x_j,x_3)\in E(D)$ for all $j\in \{4,\ldots,n\}$, otherwise, if $(x_3,x_j)\in E(D)$ implies that the induced subgraph $N_D[\{x_1,x_2,x_3,x_j\}]$ is same as $D_5$, which is a contradiction. Therefore, we can write $I(D)=x_2I_1+I_2$, which is vertex splitting of $I(D)$, where $I_1=(x_1,x_3^{w_3})$ and $I_2=x_3^{w_3}(x_4,\ldots,x_n)$. By \Cref{betti split}, we have 
       $$HS_k(I(D))=x_2HS_k(I_1)+x_2HS_{k-1}(I_2)+HS_k(I_2).$$
By \Cref{vertex splittable}, we get that $HS_1(I)$ has linear quotients. Note that $HS_k(I_1)=0$ for all $k>1$. Therefore, $HS_k(I)=x_2HS_{k-1}(I_2)+HS_k(I_2)$ for all $k\geq 2$. By \cref{tree}, $HS_k(I_2)$ has linear quotients for all $k>1$. Then by \Cref{sum}, $HS_k(I)$ has linear quotients for all $k>1$. 

\noindent
\textbf{Subcase 4(b):} Suppose $i\neq2$. Then $i\in \{4,\ldots,n\}$ and $(x_i,x_3)\in E(D)$. If $w_1>1$, then $D[\{x_1,x_2,x_3,x_i\}]$ is same as $D_6$, which is a contradiction. Hence $w_1=1$. If $w_j>1$ for some $j\in\{4,\ldots,n\}\setminus\{i\}$. Then $x_j$ not a source vertex. This gives that $(x_3,x_j)\in E(D)$. Then the induced subgraph $N_D[\{x_i,x_3,x_j\}]$ is same as $D_1$, which is a contradiction. Therefore $w_j=1$ for all $j\in\{4,\ldots,n\}\setminus\{i\}$. If $w_2>1$, then the induced subgraph $D[\{x_i,x_3,x_2\}]$ is same as $D_2$, which is a contradiction. Hence $w_2=1$. If $(x_2,x_3)\in E(D)$, then it falls Subcase-$4(a)$. Now assume $(x_3,x_2)\in E(D)$. Then $I(D)=x_2I_1+I_2$ is a vertex-splitting of $I(D)$, where $I_1=(x_1,x_3)$ and $I_2=x_3^{w_3}(x_l\mid x_l \in N_D^-(x_3))+x_3(x_l\mid x_l \in N_D^+(x_3)\setminus\{x_2\})$.  By \Cref{betti split}, we have $$HS_k(I(D))=x_2HS_k(I_1)+x_2HS_{k-1}(I_2)+HS_k(I_2).$$  By \Cref{vertex splittable}, we have $HS_1(I(D))$ has linear quotients. By \Cref{tree1}, we get $HS_k(I_2)$ has linear quotients for all $ k\geq 0$. Note that $HS_k(I_1)=0$ for all $k\geq 2$. Therefore, $HS_k(I)=x_2HS_{k-1}(I_2)+HS_k(I_2)$ for all $k\geq 2$. Then by \Cref{sum}, $HS_k(I)$ has linear quotients for all $k\geq2$. 

\noindent
\textbf{Case 5:} Suppose $w_i>1$ for some $i\in \{4,\ldots,n\}$. Since $x_i$ is not a source vertex, then we must have $(x_3,x_i)\in E(D)$. If $w_1>1$, then the induced subgraph $D[\{x_1,x_2,x_3,x_i\}]$ is same as $D_8$, which is a contradiction. Hence $w_1=1$. If $w_j>1$ for some $j\in\{4,\ldots,n\}\setminus\{i\}$. Then $x_j$ not a source vertex. This gives that $(x_3,x_j)\in E(D)$. Then the induced subgraph $N_D[\{x_i,x_3,x_j\}]$ is same as $D_2$, which is a contradiction. If $w_2>1$, then $(x_1,x_2)\in E(D)$ or $(x_3,x_2)\in E(D)$. If $(x_1,x_2)\in E(D)$ (respectively $(x_3,x_2)\in E(D)$) then the induced subgraph $D[\{x_1,x_2,x_3,x_i\}]$ (respectively $D[\{x_2,x_3,x_i\}]$ ) is same as $D_6$(respectively $D_2$), which is a contradiction. Hence $w(x_2)=1$. If $w_3>1$, then the induced subgraph $D[\{x_2,x_3,x_i\}]$ is same as $D_1$, which is contradiction. Hence $w_3=1$.  If $w(x_j)>1$, for $j\in \{4,\ldots,n\}\setminus \{i\}$ then the induced subgraph $D[\{x_i,x_3,x_j\}]$ is same as $D_2$, which is a contradiction.  Thus, $w_j=1$ for all $j\neq i$. Therefore, $I(D)=x_3I_1+I_2$ is a vertex-splitting of $I(D)$, where $I_1=(x_2,x_4,\ldots x_i^{w_i},\ldots x_n)$ and $I_2=(x_1x_2)$. Then by \Cref{betti split}, we have $HS_k(I(D))=x_3HS_k(I_1)+x_3HS_{k-1}(I_2)+HS_k(I_2)$. By \Cref{vertex splittable}, we have $HS_1(I(D))$ has linear quotients. Since  $HS_k(I_2)=0$ for all $k\geq 2$, then we get $HS_k(I(D))=x_3HS_k(I_1)$ for all $k\geq 2$. By \Cref{tree}, we have $HS_k(I_1)$ has linear quotients. Therefore $HS_k(I(D))$ has linear quotients for all $k\geq 2$.      
\end{proof}

\bibliographystyle{plain}
\bibliography{reference}
\end{document}